\documentclass[11pt,a4paper,reqno]{amsart}
\usepackage{amsthm}
\usepackage{amsmath}
\usepackage{amscd}
\usepackage{amsfonts,amssymb}
\usepackage{mathtools}
\usepackage[margin=1in]{geometry}
\usepackage{fancybox}
\usepackage{pdflscape}
\usepackage{rotating}
\usepackage{epic,eepic,bbm,longtable, enumitem, stmaryrd}
\usepackage[pagebackref,colorlinks=true,linkcolor=blue,citecolor=blue]{hyperref}
\makeatletter
\let\@wraptoccontribs\wraptoccontribs
\makeatother
\allowdisplaybreaks
\newtheorem{thm}{Theorem}[section]
\newtheorem{cor}[thm]{Corollary}

\newtheorem{lem}[thm]{Lemma}

\newtheorem{prop}[thm]{Proposition}
\theoremstyle{remark}
\newtheorem*{rem}{Remark}

\newcounter{remarkscounter}

\numberwithin{equation}{section}
\newcommand{\A}{\mathbb{A}}

\newcommand{\SL}{\mathrm{SL}}
\newcommand{\ZZ}{\mathbb{Z}}

\newcommand{\QQ}{\mathbb{Q}}

\newcommand{\lto}{\,\longrightarrow\,}
\newcommand{\lmto}{\longmapsto}
\newcommand{\OO}{\mathcal{O}}
\newcommand{\CC}{\mathbb{C}}
\newcommand{\RR}{\mathbb{R}}
\newcommand{\GG}{\mathbb{G}}

\newcommand{\quash}[1]{}

\theoremstyle{definition}

\renewcommand{\bar}{\overline}

\numberwithin{equation}{section}

\renewcommand{\hat}{\widehat}

\newcommand{\one}{\mathbbm{1}}

\newenvironment{psmatrix}
  {\left(\begin{smallmatrix}}
  {\end{smallmatrix}\right)}

\begin{document}

\title{Geometrization of summation formulae for quadrics}

\author{Chun-Hsien Hsu}
\address{Department of Mathematics\\
University of Chicago\\
Chicago, IL 60637}
\email{chunhsien@uchicago.edu}


\begin{abstract}
We geometrize the Poisson summation formula for the zero locus of a split quadratic form in an even number of variables over number fields. We do so by making explicit the relationship between Schwartz spaces on quadrics defined in two different ways: via Braverman-Kazhdan spaces and via theta lifts.
\end{abstract}

\maketitle

\tableofcontents

\section{Introduction}

Let $E$ be a number field and $\mathbb{A}_E$ be its ring of adeles. Let $\psi:E\backslash \A_E\to \CC^\times$ be a nontrivial additive character. For $\ell\ge 1,$ let $(V_\ell,Q_\ell)$ be a split quadratic space over $E$ of dimension $2\ell$. For simplicity, we fix a coordinate so that $V_\ell=\A^{2\ell}$ and $Q_\ell(v_1,\ldots,v_{2\ell})=\sum_{i=1}^{\ell} v_{2i-1}v_{2i}.$ Let $X_\ell$ be the vanishing locus of $Q_\ell$ and $X_\ell^{\circ}:=X_\ell-\{0\}$ be the smooth locus of $X_\ell$. Let $e_{2\ell}:=(0,0,\ldots,0,1)\in X_\ell^\circ(E).$

Suppose $\ell\ge 2$. There is a  minimal representation $(\sigma_\ell,\mathcal{S}(X_\ell(\A_E)))$ of $\mathrm{O}_{V_{\ell+1}}(\mathbb{A}_E).$  For $f\in \mathcal{S}(X_\ell(\A_E))$ and $g\in \mathrm{O}_{V_\ell}(\A_E),$ the integral 
$$f_{s}(g):=\int_{\A_E^\times} |t|^{s+\ell-1} f(e_{2\ell}tg)d^\times t$$  
converges absolutely for $\mathrm{Re}(s)$ large, and it extends to a meromorphic function on $\CC$. Let $P_\ell$ be the stabilizer of the line spanned by $e_{2\ell}$. The degenerate Eisenstein series
\begin{align*}
    \mathrm{Eis}(g;f_{s}):=\sum_{\xi \in P_\ell(E)\backslash \mathrm{O}_{V_\ell}(E)} f_{s}(\xi g)
\end{align*}
converges absolutely for $\mathrm{Re}(s)$ large, and defines a meromorphic function on $\CC$. For ease of notation, we write $\mathrm{Eis}(f_{s}):=\mathrm{Eis}(I_{2\ell};f_{s}).$  For $\ell\ge 3,$ it is shown in \cite{Hsu:asymptotics, Getz:Quad} that $\mathrm{Eis}(f_{s})$ can only have possible simple poles at $s=\pm1, \pm (\ell-1).$

 The element $\begin{psmatrix}
     I_{2\ell} & &\\
     & & 1\\
     & 1 &
 \end{psmatrix}\in \mathrm{O}_{V_{\ell+1}}(E)$ defines a unitary Fourier transform
\begin{align*}
    \mathcal{F}_{X_\ell}=\mathcal{F}_{X_\ell,\psi}:=\sigma_\ell\begin{psmatrix}
     I_{2\ell} & &\\
     & & 1\\
     & 1 &
 \end{psmatrix}:\mathcal{S}(X_\ell(\A_E))\longrightarrow \mathcal{S}(X_\ell(\A_E)).
\end{align*}
 Let $\zeta(s)=\zeta_E(s)$ be the completed zeta function and  $\kappa:=\kappa_E:=\mathrm{Res}_{s=1}\zeta_E(s)$. We have a Poisson summation formula on $X_\ell$.

\begin{thm}[{\cite[Theorem 1.6]{Hsu:asymptotics}}]\label{thm:Eis}
    Suppose $\ell\ge 3$. Let $f\in \mathcal{S}(X_\ell(\A_E)).$ Then
    \begin{align*}
        &\sum_{x\in X_\ell^\circ(E)} f(x)+\frac{1}{\kappa}\bigg(\mathrm{Res}_{s=-(\ell-1)} \mathrm{Eis}(f_{s})+\mathrm{Res}_{s=-1} \mathrm{Eis}(f_{s})\bigg)\\
        &=\sum_{x\in X_\ell^\circ(E)} \mathcal{F}_{X_\ell}(f)(x)+\frac{1}{\kappa}\bigg(\mathrm{Res}_{s=-(\ell-1)} \mathrm{Eis}(\mathcal{F}_{X_\ell}(f)_{s})+\mathrm{Res}_{s=-1} \mathrm{Eis}(\mathcal{F}_{X_\ell}(f)_{s})\bigg).
    \end{align*}
 Moreover, $\mathrm{Eis}(f_{s})$ is entire if $f=\otimes_v f_v$ and there are two finite places $v_1,v_2$ such that  $f_{v_1}\in C^\infty_c(X_\ell^\circ(E_{v_1}))$ and $\mathcal{F}_{X_\ell}(f_{v_2})\in C^\infty_c(X_\ell^\circ(E_{v_2})).$
\end{thm}

As explained in ibid., the residues, which we will refer to as the \textbf{boundary terms}, are related to asymptotics of $f$ toward the origin. More precisely, there are $\mathrm{O}_{V_{\ell}}(\A_E)$-equivariant boundary maps 
\begin{align*}
    \mathcal{S}(X_\ell(\A_E))\xrightarrow{\,\,(c_\ell,d_\ell)\,\,} \CC\oplus\mathcal{S}(X_{\ell-1}(\A_E)) 
\end{align*}
obtained by taking properly normalized germs of functions in $\mathcal{S}(X_\ell(\A_E))$ at the origin (see \S \ref{sec:mainproof} and its local counterparts in \S \ref{sec:BK}). Then Theorem \ref{thm:Eis} and its proof say that $\mathrm{Res}_{s=-(\ell-1)} \mathrm{Eis}(f_{s})=0$ iff $c_\ell(f)=0$ and $\mathrm{Res}_{s=-1} \mathrm{Eis}(f_{s})=0$  iff $d_\ell(f)=0.$ Nevertheless, it is unclear how to express the boundary terms geometrically when they are nonzero. 

 In the function field case, a complete geometric description of the residues in terms of the boundary maps was established in \cite{GK:auto} using the automorphy of $\sigma_\ell$ and the study of the space of automorphic linear functionals on $\mathcal{S}(X_\ell(\A_E)).$ The number field case requires a substantially different approach because of the archimedean places. A crucial input is provided by \cite{Getz:Quad}. In ibid., using theta lifts of the trivial representation of $\SL_2(\A_E)$, Getz introduced a function space $\widetilde{\mathcal{S}}(X_\ell(\A_E))$ for $\ell\ge 1$ together with a linear map $$I:\widetilde{\mathcal{S}}(X_\ell(\A_E))\longrightarrow C^\infty(X_\ell^\circ(\A_E)),$$ and established a geometric version of Theorem \ref{thm:Eis} for functions of the form $I(f)$ where $f\in \widetilde{\mathcal{S}}(X_\ell(\A_E))$. However, it is not clear if $I(f)\in \mathcal{S}(X_\ell(\A_E))$ and if the formula descends, i.e., whether it depends only on $I(f)$. In the present paper, we answer both questions affirmatively by comparing Fourier theories on $\mathcal{S}(X_\ell(E_v))$ and $\widetilde{\mathcal{S}}(X_\ell(E_v))$ over all places $v$ of $E$.

We now introduce some further notation needed to state the formula. Define $\mathcal{S}(X_1(\A_E)):=I(\widetilde{\mathcal{S}}(X_1(\A_E))).$ For $\ell= 2$, we show that there are boundary maps 
\begin{align*}
    \mathcal{S}(X_2(\mathbb{A}_E))&\xrightarrow{\,\,(c_2,d_2)\,\,} \CC\oplus\mathcal{S}(X_{1}(\A_E)).
\end{align*}
In this case an Eisenstein series may have a pole of order $2$ at $s=-1$, and $c_2$ is its (normalized) coefficient of $(s+1)^{-2}$. Due to the fact that the residue at $s=-1$ only concerns the coefficient of $(s+1)^{-1},$ there is an additional boundary map
$$a_2:\mathcal{S}(X_2(\A_E))\longrightarrow\CC$$ 
defined in \eqref{eq:adelica2}. 

Let $\CC_1$ be the representation $|\cdot|$ of $\mathbb{A}_E^\times.$ Observe that we have an automorphism of $X_1$ induced by the action of $\begin{psmatrix}
    0 & 1\\
    1 & 0
\end{psmatrix}\in \mathrm{O}_{V_1}(E)$ on $V_1$ by swapping two entries. There are two boundary maps 
\begin{align*}
    \mathcal{S}(X_{1}(\A_E))&\xrightarrow{\quad d_1\quad} \CC_1,\\
    \mathcal{S}(X_{1}(\A_E))&\xrightarrow{\quad d_1':=d_1\circ \begin{psmatrix}
    0 & 1\\
    1 & 0
\end{psmatrix}\quad} \CC_1,
\end{align*}
which correspond to asymptotics toward the origin along different axis.

Suppose $\ell\ge 2.$ For $\ell> i\ge 1,$ let
\begin{align*}
    d_{\ell,i}:=d_{i+1}\circ\cdots \circ d_\ell:\mathcal{S}(X_\ell(\A_E))\longrightarrow\mathcal{S}(X_{i}(\A_E)).
\end{align*}
Let $d_{\ell,\ell}$ be the identity map. Let $\infty$ be the set of archimedean places of $E$ and $|\infty|$ be its cardinality. Set
\begin{align*}
    \kappa':=\kappa'_E:=\frac{d}{ds}s\zeta(s)\bigg|_{s=0}.
\end{align*}
Our main result is the following formula.
\begin{thm}\label{thm:main}
    Suppose $E$ is a number field and $\ell\ge 2$.  Let $D\in \ZZ$ be the absolute discriminant of $E$. Let $f\in \mathcal{S}(X_{\ell}(\A_E)).$ Then
    \begin{align*}
        &\sum_{\xi\in X_\ell^\circ(F)}f(\xi)\\
        &+\sum_{i=3}^{\ell} |D|^{\frac{(4-\ell-i)(\ell-i)}{2}}\left(\sum_{\xi\in X_{i-1}^\circ(E)} |D|^{\frac{5}{2}-i}d_i\circ d_{\ell,i}(f)(\xi)+|D|^{i-\frac{3}{2}}\zeta(i-1)c_i\circ d_{\ell,i}(f)\right)\\
        &+2^{1-|\infty|}|D|^{-\frac{(\ell-2)^2}{2}}\left(\kappa'c_2\circ d_{\ell,2}(f)+|D|^{\frac{1}{2}}\kappa a_2\circ d_{\ell,2}(f)\right)\\
        &+\sum_{\xi\in X_1^\circ(E)}
        |D|^{-\frac{(\ell-3)(\ell-1)}{2}}d_{\ell,1}(f) + \zeta(2)|D|^{-\frac{(\ell-4)\ell}{2}}(d_1+d_1')\circ d_{\ell,1}(f)\\
       &=\sum_{\xi\in X_\ell^\circ(F)}\mathcal{F}_{X_\ell}(f)(\xi)\\
       &+\sum_{i=3}^{\ell} |D|^{\frac{(4-\ell-i)(\ell-i)}{2}}\left(\sum_{\xi\in X_{i-1}^\circ(E)} |D|^{\frac{5}{2}-i}d_i\circ d_{\ell,i}(\mathcal{F}_{X_\ell}(f))(\xi)+|D|^{i-\frac{3}{2}}\zeta(i-1)c_i\circ d_{\ell,i}(\mathcal{F}_{X_\ell}(f))\right)\\
        &+2^{1-|\infty|}|D|^{-\frac{(\ell-2)^2}{2}}\left(\kappa'c_2\circ d_{\ell,2}(\mathcal{F}_{X_\ell}(f))+|D|^{\frac{1}{2}}\kappa a_2\circ d_{\ell,2}(\mathcal{F}_{X_\ell}(f))\right)\\
        &+\sum_{\xi\in X_1^\circ(E)}
        |D|^{-\frac{(\ell-3)(\ell-1)}{2}} d_{\ell,1}(\mathcal{F}_{X_\ell}(f))+\zeta(2)|D|^{-\frac{(\ell-4)\ell}{2}}(d_1+d_1')\circ d_{\ell,1}(\mathcal{F}_{X_\ell}(f)).
    \end{align*}
\end{thm}

\begin{rem}
    Our boundary maps $c_\ell,d_\ell$ are independent of $\psi,$ while
  boundary maps $d_\ell$  in \cite{GK:auto}, denoted as $\mathcal{B}$ therein, are normalized so that spherical vectors are mapped to spherical vectors, and hence depend on $\psi.$
\end{rem}

By comparing terms between Theorem \ref{thm:Eis} and Theorem \ref{thm:main} for $\ell\ge 3$, we obtain a geometric expression of residues of Eisenstein series:
\begin{align*}
    \mathrm{Res}_{s=-(\ell-1)} \mathrm{Eis}(f_{s})=\kappa|D|^{\ell-\frac{3}{2}}\zeta(\ell-1)c_\ell(f),
\end{align*}
and $\kappa^{-1}\mathrm{Res}_{s=-1}\mathrm{Eis}(f_{s})$ is equal to the sum of the rest of the terms involving $d_\ell$. This expression is akin to regularized Siegel-Weil identities \cite{Ichino:SW,Ikeda:SW,KR:Siegel-Weil}. This indicates that regularized Siegel-Weil identities can be used to geometrize boundary terms at least for Braverman-Kazhdan spaces. For instance, one can rephrase the Poisson summation formulae on the Lagrangian Grassmannian in \cite{Getz:Liu:BK} in a geometric fashion via regularized Siegel-Weil identities. 

The paper is structured as follows. We set up our notations in \S \ref{sec:prelim}. In \S \ref{sec:BK}  we review the local theory of Schwartz spaces on $X_\ell$ over local fields for $\ell\ge 3$ constructed in \cite{Getz:Hsu:Leslie} by viewing $X_\ell$ as Braverman-Kazhdan spaces. Then we review in \S\ref{sec:thetalift} the Fourier theory on $X_\ell$ for general $\ell\ge 1$ defined via theta lift by \cite{Getz:Quad}. We show that the two theories coincide in \S \ref{sec:agree} and \S \ref{sec:agree:arch}. Then we summarize our results in the adelic setting in \S \ref{sec:mainproof} and prove Theorem \ref{thm:main}.

\section{Preliminaries}\label{sec:prelim}

Throughout the paper $F$ denotes a local field of characteristic zero. We let $|\cdot|$ be the number-theorist's norm on $F$. Thus $|\cdot|$ is the usual Euclidean norm if $F=\RR,$ and $|z|=z\bar{z}$ if $F=\CC$. When $F$ is nonarchimedean, we denote by $\OO=\OO_F$ the ring of integers of $F$ and fix once and for all a choice of uniformizer $\varpi$. Then $q=|\varpi|^{-1}$ is the cardinality of the residue field $\OO/\varpi$. 

Vectors in $F^n$ are row vectors. For nonarchimedean $F,$ let $|\cdot|$ be the box norm on $F^n$ given by
\begin{align*}
    |(v_1,\ldots,v_n)|:=\max_{1\le i\le n} |v_i|.
\end{align*}
When $F$ is archimedean, let $|\cdot|_\RR$ be the usual Euclidean norm on $F^n$ given by
\begin{align*}
     |(v_1,\ldots,v_n)|_\RR:=\left(\sum_{i=1}^n v_i\overline{v}_i\right)^{\frac{1}{2}},
\end{align*}
and let $|\cdot|:=|\cdot|_{\RR}^{[F:\RR]}.$ For a subset $X$ of $F^n,$ let $X^1$ be the set of vectors in $X$ of norm $1$.

\subsection{Measure}

Given a nontrivial additive character $\psi:F\to \CC^\times,$ the Haar measure $dt$ on $F$ will always be normalized so that the Fourier transform on $\mathcal{S}(F)$ defined by $\psi$ is self-dual. Let
\begin{align*}
    \zeta(s):=\begin{cases}
        \frac{1}{1-q^{-s}} & \textrm{if } F \textrm{ is nonarchimedean,}\\
        \pi^{-s/2}\Gamma(s/2) & \textrm{if } F=\RR,\\
        2(2\pi)^{-s}\Gamma(s) &\textrm{if } F=\CC,
    \end{cases}
\end{align*}
be the local zeta function. We set $d^\times t:=\frac{\zeta(1)dt}{|t|}$. It is a Haar measure on $F^\times$. Let $K_{\GG_m}$ be the maximal compact subgroup of $F^\times.$ We denote by $\mathrm{vol}(K_{\GG_m})$ the constant such that for $K_{\GG_m}$-invariant functions $f\in \mathcal{S}(F^\times)$
\begin{align*}
    \int_{F^\times} f(t)d^\times t=\mathrm{vol}(K_{\GG_m})\int_{F^\times/K_{\GG_m}} f(r)dr,
\end{align*}
where $dx$ is the counting measure on $F^\times/K_{\GG_m}=\ZZ$ when $F$ is nonarchimedean, and $rdr$ is the Lebesgue measure on $F^\times/K_{\GG_m}=\RR_{>0}$ when $F$ is archimedean.

Let $T<B$ be the group of diagonal matrices and the group of upper triangular matrices in $\SL_2$, respectively. Let $N$ be the unipotent radical of $B$. Since $T\cong \GG_m$ and $N\cong \GG_a$, we equip $T(F)$ and $N(F)$ with measures
transferred from those on $F$ and $F^\times$. Let 
\begin{align*}
    K:=\begin{cases}
        \SL_2(\OO) & \textrm{if } F \textrm{ is nonarchimedean,} \\
        \mathrm{SO}_2(\RR) & \textrm{if } F=\RR,\\
        \mathrm{SU}_2(\RR) &\textrm{if } F=\CC.
    \end{cases}
\end{align*}
We normalize the Haar measure on $\SL_2(F)$ so that by the Iwasawa decomposition
\begin{align*}
    d\left(\begin{psmatrix}
        1 & n\\
          & 1
    \end{psmatrix}\begin{psmatrix}
        t & \\
          & t^{-1}
    \end{psmatrix}k\right)=dn\frac{d^\times t}{|t|^2}dk,
\end{align*}
where $dk(K)=1$. This choice of Haar measure agrees with that in \cite{Getz:Quad}.

\subsection{Schwartz space}\label{ssec:Schwartz}

For each real algebraic variety $X$ (over $\RR$), a Schwartz space $\mathcal{S}_{\mathrm{ES}}(X)$ is defined in \cite[Definition 3.7]{Elazar:Shaviv} (generalizing the previous work in \cite{AG:Nash}). In the case $X$ is affine, one can embed $X$ as a closed subset of $\RR^n$ in the category of real algebraic varieties. Then
\begin{align*}
    \mathcal{S}_{\mathrm{ES}}(X):= \mathcal{S}(\RR^n)/I=\mathcal{S}(\RR^n)|_{X}
\end{align*}
where $\mathcal{S}(\RR^n)$ is the usual space of Schwartz functions on $\RR^n,$ and $I \le \mathcal{S}(\RR^n)$ is the (closed) ideal of functions that vanish identically on $X$. The space $\mathcal{S}_{\mathrm{ES}}(X)$ endowed with the quotient topology is a nuclear Fr\'echet space. It does not depend on the choice of embeddings. 

Let $X$ be a quasi-projective $F$-variety with nonempty $X(F)$. When $F$ is archimedean, we let
\begin{align*}
    \mathcal{S}_{\mathrm{ES}}(X(F)):=\mathcal{S}_{\mathrm{ES}}(\mathrm{Res}_{F/\RR}X(\RR)).
\end{align*}
When $F$ is nonarchimedean, following the notation in the archimedean case, we define 
$$\mathcal{S}_{\mathrm{ES}}(X(F)):=C^\infty_c(X(F))$$ to be the space of locally constant functions on $X(F)$ with compact support. If $X$ is affine, then for any closed $F$-embedding $X\lhook\joinrel\xrightarrow{\,\,} \mathbb{A}^n,$ we have
\begin{align*}
    \mathcal{S}_{\mathrm{ES}}(X(F))=\mathcal{S}(F^n)\big|_{X(F)}.
\end{align*}
In any case, when $X$ is a smooth variety, we write
\begin{align*}
    \mathcal{S}(X(F)):=\mathcal{S}_{\mathrm{ES}}(X(F)).
\end{align*}
For singular $X$, we only define $\mathcal{S}(X(F))$ case-by-case. When $F$ is archimedean, $\mathcal{S}(X(F))$ in this paper is always a nuclear Fr\'echet space.

Let $E$ be a number field and $X$ be an affine $E$-variety. Assume $\mathcal{S}(X(E_v))$ is defined for all places $v$ and a basic function $b_v \in \mathcal{S}(X(E_v))$
is chosen for almost all $v$. Let $S$ be a finite set of places of $E$. If $S$ contains all infinite places, we define $$\mathcal{S}(X(\A_{E}^S)):=\bigotimes_{v\not\in S}{}^{\prime}
\mathcal{S}(X(E_v))$$ to be the restricted
tensor product taken with respect to $b_v$. If $S$ is a subset of infinite places of $E$, then
$\mathcal{S}(X(E_S)) := \widehat{\bigotimes}_{v\in S}\mathcal{S}(X(E_v))$ is the completed projective topological tensor product. For
general $S$, we put
\begin{align*}
\mathcal{S}(X(\A_E^S
)) := \mathcal{S}(X(E_{\infty-S}))\otimes \mathcal{S}(X(\A_E^{\infty\cup S}
)).
\end{align*}

\subsection{A brief review of Tate's thesis}\label{ssec:Tate}

For $f\in \mathcal{S}(F),$ let $\hat{f}$ be the Fourier transform of $f$. For $s\in \CC,$ the zeta integral
\begin{align*}
    Z(f,s):=\int_{F^\times} f(t)|t|^s d^\times t
\end{align*}
converges absolutely for $\mathrm{Re}(s)>0,$ and $\zeta(s)^{-1}Z(f,s)$ is an entire function satisfying a functional equation
\begin{align}\label{eq:Tate}
    \frac{Z(\hat{f},1-s)}{\zeta(1-s)}=\varepsilon(s,\psi)\frac{Z(f,s)}{\zeta(s)}.
\end{align}
Here $\varepsilon(s,\psi)$ is the Tate $\varepsilon$-factor (attached to the trivial character); it is an entire function of the form $Ae^{Bs}$ for some complex numbers $A,B$. It is the constant function $1$ when $F$ is nonarchimedean and $\psi$ is unramified, i.e., $\psi$ has conductor $\OO,$ or when $F$ is archimedean and $\psi(t)=e^{2\pi i\mathrm{tr}_{F/\RR}(t)}.$

Let $E$ be a number field and $\psi=\otimes \psi_v:E\backslash \A_E\to \CC^\times$ be a nontrivial additive chracter. Then $\varepsilon(s):=\prod_{v} \varepsilon(s,\psi_v)$ is independent of the choice of $\psi,$ and $\varepsilon(s)=|D|^{\frac{1}{2}-s}$ where $D\in \ZZ$ is the absolute discriminant of $E$. The (completed) zeta function $\zeta(s)=\prod_{v}\zeta_v(s)$ satisfies the functional equation
\begin{align*}
    \zeta(s)=\varepsilon(s)\zeta(1-s).
\end{align*}

\subsection{Odd-dimensional cones}

Let $E$ be a field of characteristic zero. Let $V_0$ be the zero vector space. For a positive integer $\ell,$ let $(V_\ell,Q_\ell)$ be a split quadratic space over $E$ of dimension $2\ell$. For simplicity, we identify $V_\ell$ with $\A^{2\ell}$ so that under the standard basis $\{e_i\}_{1\le i\le 2\ell}$ for $v=(v_1,\ldots, v_{2\ell})\in V_\ell(F)$
\begin{align*}
    Q_\ell(v)=\sum_{i=1}^{\ell} v_{2i-1}v_{2i}.
\end{align*}
The associated bilinear form is $ \langle v,w\rangle=vJ_\ell w^t,$ where
\begin{align*}
        J_\ell:=\begin{psmatrix}
             J & & &\\
               & J & &\\
               & & \ddots & \\
               & & & J
        \end{psmatrix},\quad \textrm{  and } J:=\begin{psmatrix}
         & 1\\
         1 &   
        \end{psmatrix}.
\end{align*}

Let $\mathrm{O}_{V_\ell}$ be the orthogonal group of $(V_\ell,Q_\ell)$. Let $P_{\ell+1}$ be the stabilizer of the isotropic line spanned by $e_{2\ell+2}$ in $V_{\ell+1}$, and let $M_{\ell+1}$ be a Levi subgroup of $P_{\ell+1}$ given on $R$-points by
\begin{align*}
    M_{\ell+1}(R):=\left\{\begin{psmatrix}
        h &  &\\
          & a &\\
          &  & a^{-1}
    \end{psmatrix}: h\in \mathrm{O}_{V_{\ell}}(R), a\in R^\times\right\}.
\end{align*}
We identify $\mathrm{O}_{V_{\ell}}$ as a subgroup of $M_{\ell+1}$ and hence as a subgroup of $\mathrm{O}_{V_{\ell+1}}$. Let $N_{\ell+1}$ be the unipotent radical of $P_{\ell+1}$. We fix an isomorphism of scheme
\begin{align*}
        u:V_{\ell}&\longrightarrow N_{\ell+1}\\
            v    &\mapsto \begin{psmatrix}
                I_{2\ell} & & J_\ell v^t\\
                    -v       &  1 & Q_\ell(v)\\
                         &  & 1
            \end{psmatrix}.
\end{align*}
We note that our map differs from that in \cite[(2.7)]{Getz:Quad}, since vectors there are taken to be column vectors.

Let $X_\ell$ be the vanishing locus of $Q_\ell$ and $X_\ell^\circ$ be its smooth locus. For $\ell\ge 1$, the origin is the unique singularity of $X_\ell$ so  $X_\ell^\circ=X_\ell-\{0\}$. Let $\mathrm{SO}_{V_\ell}<\mathrm{O}_{V_\ell}$ be the special orthogonal group. For $\ell\ge 2,$ the natural map
\begin{align*}
     \mathrm{SO}_{V_\ell}&\lto X_\ell \\
     g &\longmapsto e_{2\ell}g
\end{align*}
induces an isomorphism of schemes $P_\ell^{\mathrm{der}}\backslash \mathrm{SO}_{V_\ell} \cong X_\ell^\circ$, so  $X_{\ell}$ is the affine closure of $X_\ell^\circ$. We will often identify these spaces.

\section{Fourier theory on $X_\ell$ $(\ell\ge 3)$ as Braverman-Kazhdan spaces}\label{sec:BK}

Let $\ell\ge 3$. Let $\psi:F\to \CC^\times$ be a nontrivial additive character. A (unitary) Fourier transform
\begin{align*}
    \mathcal{F}_{X_{\ell}}=\mathcal{F}_{X_{\ell},\psi}:L^2(X_\ell(F))\lto L^2(X_\ell(F))
\end{align*}
is constructed in \cite{BKnormalized}. Their construction is refined in \cite{Getz:Hsu:Leslie}. The Fourier transform has order $2$, i.e.,
 $\mathcal{F}_{X_\ell}\circ \mathcal{F}_{X_\ell}=\mathrm{Id}.$ Therefore, $\mathcal{F}_{X_{\ell}}$ descends to an automorphism of the Schwartz space 
\begin{align*}
    \mathcal{S}(X_{\ell}(F)):=\mathcal{S}(X_{\ell}^\circ(F))+\mathcal{F}_{X_{\ell}}
(\mathcal{S}(X_\ell^\circ(F)))\subset C^
\infty(X_{\ell}^\circ(F)).
\end{align*}
When $F$ is archimedean, $\mathcal{S}(X_{\ell}(F))$ is equipped with a (nuclear) Fr\'echet topology.
By the work of \cite{Getz:Hsu:Leslie, Hsu:asymptotics} for $f\in \mathcal{S}(X_{\ell}(F))$ and $y\in X_{\ell}^\circ(F)$
\begin{align*}
    \mathcal{F}_{X_{\ell}}(f)(y)=\int_{F^\times} |t|^{\ell-2}\psi(t^{-1}) \left(\int_{X_{\ell}^\circ(F)} f(x)\psi(t\langle x,y\rangle)dx\right) \frac{d^\times t}{\zeta(1)}
\end{align*}
for a properly normalized Radon measure $dx$ on $X_{\ell}^\circ(F)$. When $F$ is nonarchimedean, the same formula is obtained in \cite{GK:cone} by realizing $\mathcal{S}(X_{\ell}(F))$ as the (unitarizable) minimal representation $\sigma_\ell$ of $\mathrm{O}_{V_{\ell+1}}(F)$ contained in the degenerate principal series $\mathrm{Ind}_{P_{\ell+1}}^{\mathrm{O}_{V_{\ell+1}}}(1_{-1}).$ The action of $\mathrm{O}_{V_{\ell+1}}(F)$ on $\mathcal{S}(X_{\ell}(F))$ can be explicitly described as follows \cite{GK:cone}: for $f\in \mathcal{S}(X_\ell(F))$ and $x\in X_{\ell}^\circ(F)$
\begin{align}\label{eq:minimal:act}  
    \begin{split}
    \sigma_\ell(h)f(x)&=f(xh) \quad\quad\quad\quad\quad\quad\quad  \textrm{for } h\in \mathrm{O}_{V_\ell}(F),\\
       \sigma_\ell\begin{psmatrix}
            I_{2\ell} & &\\
            & a &\\
            & & a^{-1}
        \end{psmatrix}f(x)&=|a|^{\ell-1} f(ax) \quad\quad\quad\quad\,\,\,\, \textrm{for }a\in F^\times,\\
        \sigma_\ell(u(v))f(x)&=\overline{\psi}(\langle v,x\rangle)f(x) \quad\quad\quad\,\,\,\,\, \textrm{for }v\in V_{\ell}(F),\\
        \sigma_\ell \begin{psmatrix}
            I_{2\ell} & &\\
             & & 1\\
             &1 &
        \end{psmatrix}f&=\mathcal{F}_{X_\ell}(f).
    \end{split}
\end{align}
The statement is also valid when $F$ is archimedean.  This is proved in \cite[Example 7.10]{Hsu:asymptotics} by computing generators of the Weyl algebra on $X_\ell$ and comparing them with the Bessel operators defined in \cite{Kobayashi:Mano, minrep:real}.

When $\ell=2$,  $X_2$ is not a Braverman-Kazhdan space considered in \cite{Getz:Hsu:Leslie} or \cite{BKnormalized}. We define $\mathcal{S}(X_{2}(F))\subset L^2(X_{2}(F))$ to be the space of smooth vectors in the minimal representation $\sigma_2$ of $\mathrm{O}_{V_3}(F)$ (also of $\SL_4(F)$) that is a subrepresentation of $\mathrm{Ind}_{P_3}^{\mathrm{O}_{V_3}}(1_{-1})$. We refer one to  \cite{SW:minimal} and \cite{minrep:real} for details. We denote by $\mathcal{F}_{X_{2}}$ the action of $\begin{psmatrix}
     & & 1&\\
    & &  & 1\\
    1&  & & \\
    & 1 &  & 
\end{psmatrix}\in \SL_4(F)$ under the minimal representation. The action of $\sigma_2$ is analogous to \eqref{eq:minimal:act} above. When $F$ is nonarchimedean, $\mathcal{S}(X_{2}(F))=\mathcal{S}(X_{2}^\circ(F))+\mathcal{F}_{X_{2}}\left(\mathcal{S}(X_{2}^\circ(F))\right)$ by the argument in \cite[\S 2.2]{GK:cone}.

\begin{thm}\label{thm:asymplge3}
For $\ell \ge 3,$ the $\mathrm{O}_{V_\ell}(F)$-module $\mathcal{S}(X_\ell(F))/\mathcal{S}(X_\ell^\circ(F))$ is isomorphic to 
\begin{align*}
    \begin{cases}
        \CC\oplus\mathcal{S}(X_{\ell-1}(F)) & \textrm{if } F \textrm{ is nonarchimedean,}\\
        \bigg(\CC\oplus \mathcal{S}(X_{\ell-1}(\RR))\bigg)\CC[[X_\ell]] & \textrm{if } F=\RR \textrm{ and } 2\nmid \ell,\\
        \bigg(\CC+ \mathcal{S}(X_{\ell-1}(F))+\CC(-\log |x|)\bigg)\CC[[\mathrm{Res}_{F/\RR}X_\ell]]& \textrm{otherwise}.
    \end{cases}
\end{align*}
In particular, $\mathcal{S}_{\mathrm{ES}}(X_\ell(F))\subseteq \mathcal{S}(X_\ell(F))$ and there are $\mathrm{O}_{V_\ell}(F)$-equivariant maps
\begin{align*}
    \mathcal{S}(X_\ell(F))/\mathcal{S}(X_\ell^\circ(F))\xrightarrow{(c_\ell,d_\ell)} \CC\oplus \mathcal{S}(X_{\ell-1}(F)). 
\end{align*}
\end{thm}

\begin{proof}
    When $F$ is nonarchimedean, this is proved in \cite{GK:cone} and \cite[Example 4.4.2]{Hsu:asymptotics}. For archimedean $F,$ this is computed in \cite[Example 5.4.2 and \S 5.5]{Hsu:asymptotics}.
\end{proof}

These maps are obtained by taking germs of functions at the origin (the singularity).  The notation $-\log |x|$ above is symbolic. It is used to keep track of the asymptotics. We remark that when $F=\RR$ and $2\mid \ell$ or $F=\CC$, the map $c_\ell$ takes image in $\CC(-\log |x|)$ instead of $\CC$. The maps $c_\ell$ are canonical and hence independent of the choice of $\psi$. The maps $d_\ell$ are only determined up to a scalar since we have chosen a model.

We now choose a normalization of $d_\ell$ that is independent of $\psi$. Let $\ell\ge 2$. When $F$ is nonarchimedean,  for $x\in X_\ell^\circ(F)$ let
\begin{align}\label{eq:basicnonarch}
    b_\ell(x):=\sum_{j=0}^\infty q^{j(\ell-2)}\one_{V_\ell(\OO)}\left(\frac{x}{\varpi^j} \right).
\end{align}
By Theorem \cite[Theorem 6.1]{SW:minimal}, if $\psi$ is unramified, $b_\ell$ is the spherical vector of $\sigma_\ell$ with $b_\ell(x)=1$ for $x\in X_\ell^\circ(\mathcal{O}).$
When $F$ is archimedean, consider the normalized  $K$-Bessel function
\begin{align*}
    \widetilde{K}_{\nu}(z):=\left(\frac{z}{2}\right)^{-\nu}K_{\nu}(z)
\end{align*}
defined in \cite[\S 7.2]{Kobayashi:Mano} for $\nu\ge 0$. By \cite[(2.14) and Theorem 2.19]{minrep:real}, there is a constant $c\in \RR_{>0}$ (depending on $F$ and $\psi$) such that $\widetilde{K}_{[F:\RR]\frac{\ell-2}{2}}(c|x|_\RR)$ is a spherical vector of $\sigma_\ell$. For $x\in X_\ell^\circ(F),$ let
\begin{align}\label{eq:basicarch}
    b_\ell(x):=2^{[F:\RR]}([F:\RR]\pi)^{[F:\RR]\frac{\ell-2}{2}}\widetilde{K}_{[F:\RR]\frac{\ell-2}{2}}\bigg(2\pi[F:\RR] |x|_\RR\bigg).
\end{align}
For $\ell\ge 3,$ we normalize $d_\ell$ so that \begin{align}\label{eq:dnormalize}
    d_{\ell}(b_\ell)=b_{\ell-1}.
\end{align}

The space $\mathcal{S}(X_\ell^\circ(F))$ is stable under $\sigma_\ell(P_{\ell+1}(F))$.
\begin{lem}\label{lem:Iact} 
Suppose $F$ is nonarchimedean. For $\ell\ge 2,$ the $P_{\ell+1}(F)$-module $(\sigma_\ell,\mathcal{S}(X_\ell^\circ(F)))$ is irreducible.
\end{lem}

\begin{proof}
    Since $\mathrm{O}_{V_\ell}(F)$ acts transitively on $X_\ell^\circ(F)$ and $C^\infty(X_\ell(F)^1)\subset \mathcal{S}(X_\ell^\circ(F))$ is stable under $N_{\ell+1}(F)$, it suffices to show $C^\infty(X_\ell(F)^1)$ is an irreducible $N_{\ell+1}(F)\rtimes \mathrm{O}_{V_\ell}(\OO)$-module. Let $0  \neq f \in C^\infty(X_{\ell}(F)^1)$ and let $W$ be the vector space spanned by the $N_{\ell+1}(F)\rtimes \mathrm{O}_{V_\ell}(\OO)$-translates of $f.$ Let $x \in X_{\ell}(F)^1.$  
    The group $\mathrm{O}_{V_\ell}(\OO)$ acts transitively on $X_{\ell}(F)^1.$
    Thus using the action of $\mathrm{O}_{V_\ell}(\OO)$ there is an element $f' \in W$ such that $f'(x) \neq 0.$     By Fourier inversion, the functions
    \begin{align*}
        \xi\mapsto \psi(v\xi^t),\quad \xi\in X_\ell(F)^1 
    \end{align*}
    ranging over all $v\in V_{\ell}(F)$ generate $C^\infty(V_\ell(\OO))|_{X_\ell(F)^1}$.
    Thus $\one_{x+\varpi^n V_\ell(\OO)}|_{X_\ell^1(F)}\cdot f' \in W$ for all $n.$  We deduce that $W$ contains $\one_{x+\varpi^nV_{\ell}(\OO)}$ for all $n$ sufficiently large and all $x.$  Thus $W=C^\infty(X_\ell(F)^1).$ 
\end{proof}

\section{Fourier theory on $X_\ell$ \`a la theta lift}\label{sec:thetalift}

 For $\ell\ge 1,$ let 
\begin{align*}  \rho_\ell:=\rho_{\ell,\psi}:\SL_2(F)\times\mathcal{S}
(V_\ell(F))\to \mathcal{S}
(V_\ell(F))\end{align*}
be the Weil representation. Explicitly, it is determined by
\begin{align*}
    \rho_\ell\begin{psmatrix}
        0 & 1\\ 
        -1 & 0
    \end{psmatrix}f(v)&=\int_{V_\ell(F)} f(u)\psi(\langle u,v\rangle)du,\\
    \rho_\ell\begin{psmatrix}
        1 & t \\
        0 & 1
    \end{psmatrix}f(v)&=\psi(tQ_\ell(v))f(v),\\
    \rho_\ell \begin{psmatrix}
        a & 0\\ 
        0 & a^{-1}
    \end{psmatrix}f(v)&=|a|^\ell f(av).
\end{align*}
Here the measure $du$ on $V_\ell(F)$ is the self-dual measure with
respect to $\langle \cdot,\cdot\rangle$ and $\psi$. The representation commutes with the natural (right) action of $\mathrm{O}_{V_\ell}$ on $V_\ell$.

Let $R$ be the natural right action of $\SL_2$ on $\A^2$. Then we have a smooth representation
\begin{align*}
    r_\ell:=\rho_\ell\otimes R:\SL_2(F)\times\mathcal{S}
(V_\ell(F)\oplus F^2)\to \mathcal{S}
(V_\ell(F)\oplus F^2).
\end{align*}
Define the space of coinvariants
\begin{align*}
    \widetilde{\mathcal{S}}(X_\ell(F)):=\mathcal{S}(V_\ell(F)\oplus F^2)_{r_\ell(\SL_2(F))}.
\end{align*} 
When $F$ is archimedean,  coinvariants are taken topologically so that $\widetilde{\mathcal{S}}(X_\ell(F))$ is a (nuclear) Fr\'echet space under the quotient topology.

For $f\in \mathcal{S}
(V_\ell(F)\oplus F^2)$ and $v\in X_\ell^\circ(F)$, define the integral
\begin{align*}
    I(f)(v):=\int_{N(F)\backslash \SL_2(F)} r_\ell(g)f(v,0,1)\, d\dot{g}.
\end{align*}
It is absolutely convergent by \cite[Lemma 3.1]{Getz:Quad} and defines a linear operator
\begin{align*}
    I:\widetilde{\mathcal{S}}(X_{\ell}(F))\lto C^\infty(X_\ell^\circ(F))
\end{align*}
that is continuous when $F$ is archimedean. The same lemma implies for $\ell \geq 2$ one has $I(\widetilde{\mathcal{S}}(X_{\ell}(F))) \subset  L^2(X_{{\ell}}(F)).$

On $\mathcal{S}(F^2)$ we have the partial Fourier transform in the second variable
\begin{align*}
    \mathcal{F}_2:\mathcal{S}(F^2)&\lto\mathcal{S}(F^2)\\
    f&\longmapsto \bigg((u_1,u_2)\mapsto \int_{F} f(u_1,x)\psi(u_2 x)\, dx\bigg).
\end{align*}
We will constantly identify $V_{\ell+1}$ with $V_\ell\times \A^2$. This induces a Fourier transform
\begin{align*}
    \mathcal{F}_2:=\mathrm{id}_{\mathcal{S}(V_\ell(F))}\otimes \mathcal{F}_2:\mathcal{S}(V_{\ell+1}(F))\lto \mathcal{S}(V_{\ell}(F)\oplus F^2).
\end{align*}
By \cite[Lemma 3.2]{Getz:Quad} it satisfies the property  
$\mathcal{F}_2\circ \rho_{\ell+1}(g)=r_\ell(g)\circ \mathcal{F}_2$ for all $g\in \SL_2(F)$. In particular, it allows us to view $\widetilde{\mathcal{S}}(X_{\ell}(F))$ as a smooth representation $\tilde{\sigma}_\ell$ of $\mathrm{O}_{V_{\ell+1}}(F),$ which is the big theta lift of the trivial representation of $\SL_2(F)$.

Consider the $\SL_2(F)$-equivariant Fourier transform:
\begin{align*}
    \mathcal{F}_{\wedge}: \mathcal{S}(F^2)&\lto \mathcal{S}(F^2)\\
    f&\mapsto \bigg( (u_1,u_2)\mapsto \int_{F^2}f(w_1,w_2)\psi(w_1u_2-w_2u_1) \,dw_1dw_2\bigg).
\end{align*}
It induces an automorphism
\begin{align*}
    \widetilde{\mathcal{F}}_{X_\ell}:=\mathrm{id}_{\mathcal{S}(V_\ell(F))}\otimes \mathcal{F}_{\wedge}:\widetilde{\mathcal{S}}(X_{\ell}(F))\lto \widetilde{\mathcal{S}}(X_\ell(F))
\end{align*}
which is continuous when $F$ is archimedean.

The action of $\tilde{\sigma}_\ell$ is computed in \cite[Proposition 3.3]{Getz:Quad}.

\begin{prop}\label{prop:action}
    Let $f\in \widetilde{\mathcal{S}}(X_\ell(F))$ and $x\in X_\ell^\circ(F)$. For $h\in \mathrm{O}_{V_\ell}(F),a\in F^\times$ and $v\in V_\ell(F),$ one has
    \begin{align*}
        I(\tilde{\sigma}_\ell(h)f)(x)&=I(f)(xh),\\
        I\left(\tilde{\sigma}_\ell\begin{psmatrix}
            I_{2\ell} & &\\
            & a &\\
            & & a^{-1}
        \end{psmatrix}f\right)(x)&=|a|^{\ell-1} I(f)(ax),\\
        I(\tilde{\sigma}_\ell(u(v))f)(x)&=\overline{\psi}(\langle v,x\rangle )I(f)(x),\\
        \tilde{\sigma}_\ell \begin{psmatrix}
            I_{2\ell} & &\\
             & & 1\\
             &1 &
        \end{psmatrix}f&=\widetilde{\mathcal{F}}_{X_\ell}(f).
    \end{align*}\qed
\end{prop}
\noindent In particular, we have 
\begin{align}\label{eq:Pequal}
    I(\tilde{\sigma}_\ell(g)(f))=\sigma_\ell(g)I(f),\quad \forall g\in P_{\ell+1}(F).
\end{align}

When $F$ is nonarchimedean and $\psi$ is unramified, for $\ell \geq 2$ we have $b_{\ell}=I(\one_{V_{\ell+1}(\OO)})$ by \cite[Lemma 3.8]{Getz:Quad}. For $\ell=1,$ we define for $x\in X_1^\circ(F)$
\begin{align}\label{eq:basicagree}
    b_1(x):=I(\one_{V_{2}(\OO)})(x)=\sum_{j=0}^\infty q^{-j}\one_{V_\ell(\OO)}\left(\frac{x}{\varpi^j} \right).
\end{align}
Suppose $F$ is archimedean.  Consider the function 
\begin{align*}
    \tilde{b}_{\ell}(v):=e^{-[F:\RR]\pi |v|_\RR^2},\quad v\in V_{\ell+1}(F).
\end{align*}
 When $\psi(t)=e^{2\pi i\mathrm{tr}_{F/\RR}(t)},$ $\tilde{b}_\ell$ is invariant under the action of $r_\ell(K)$. If $F=\RR,$ by the Iwasawa decomposition for $x\in X_\ell^\circ(F)$
\begin{align*}
    I(\tilde{b}_\ell)(x)&=\int_{F^\times} |t|^{\ell-2}\tilde{b}_\ell(tx,0,t^{-1})d^\times t\\
    &=2\int_{0}^\infty t^{\ell-2}e^{-\pi t^2|x|_\RR^2-\pi t^{-2}}\frac{dt}{t}\\
    &=\int_{0}^\infty t^{\frac{\ell-2}{2}}e^{-\pi t|x|_\RR^2-\pi t^{-1}}\frac{dt}{t}\\
    &=\int_{0}^\infty t^{-\frac{\ell-2}{2}}e^{-\pi t^{-1}|x|_\RR^2-\pi t}\frac{dt}{t}.
\end{align*}
By \cite[\S 7.12 (23)]{Highertrans:II} it equals
\begin{align*}
    2|x|_\RR^{\frac{2-\ell}{2}}K_{\frac{\ell-2}{2}}( 2\pi |x|_\RR)=2\pi^{\frac{\ell-2}{2}}\widetilde{K}_{\frac{\ell-2}{2}}(2\pi |x|_\RR).
\end{align*}
Similarly when $F=\CC$ we have
\begin{align*}
  I(\tilde{b}_\ell)(x)=4(2\pi)^{\ell-2}\widetilde{K}_{\ell-2}(4\pi |x|_\RR).
\end{align*}
Therefore, $b_\ell=I(\tilde{b}_\ell)$ for $\ell\ge 2,$ and we let $b_1:=I(\tilde{b}_1).$ Note that for any $F$ and $\ell,$ $\widetilde{\mathcal{F}}_{X_\ell}(\tilde{b}_\ell)=\tilde{b}_\ell.$

Define for $f\in \mathcal{S}(V_\ell(F)\oplus F^2)$
\begin{align*}
    Z_{\ell}(f,s):=\int_{N(F)\backslash \SL_2(F)} e^{H(g)(2-\ell-s)}r_\ell(g)f(0_{\underline{2\ell}},0,1)d\dot{g},
\end{align*} 
where 
\begin{align*}
    H(g):=\log|t|,\quad g=\begin{psmatrix}
        1 &  n\\
         & 1
    \end{psmatrix}\begin{psmatrix}
        t & \\
         & t^{-1}
    \end{psmatrix}k\in N(F)T(F)K.
\end{align*}
For $t\in F^\times,$ let 
\begin{align*}
    \Psi_f(t):=\int_{K} r_\ell(k)f(0_{\underline{2\ell}},0,t)dk.
\end{align*}
Then 
\begin{align*}
     Z_{\ell}(f,s)=Z(\Psi_f,s)=\int_{F^\times}\Psi_f(t)|t|^s d^\times t
\end{align*}
is a Tate integral considered in \S \ref{ssec:Tate}, so it extends to a meromorphic function on $\CC$ such that $\zeta(s)^{-1}Z_\ell(f,s)$ is entire. Define a linear functional
\begin{align*}
    \tilde{c}_\ell:\mathcal{S}(V_{\ell}(F)\oplus F^2)&\lto \CC\\
    f&\mapsto \begin{cases}
        Z_{\ell}(f,2-\ell) & \textrm{if } \ell\neq 1 \textrm{ and } \zeta(s) \textrm{ is holomorphic at } s=2-\ell,\\
        \mathrm{Res}_{s=2-\ell}Z_{\ell}(f,s) & \begin{aligned}
&\textrm{if } \zeta(s) \textrm{ has a pole at } s=2-\ell,  \\
&\textrm{and } F \textrm{ is archimedean},
\end{aligned}\\
\zeta(s)^{-1}Z_2(f,s)\bigg|_{s=0} & \textrm{if } \ell=2 \textrm{ and } F \textrm{ is nonarchimedean},\\
        \zeta(1)^{-1}Z_{1}(f,1) &\textrm{if } \ell=1.
    \end{cases}
\end{align*}

\begin{lem}
    The functional $\tilde{c}_\ell$ is $r_\ell(\SL_2(F))$-invariant. 
\end{lem}
\begin{proof}
   Let $f\in \mathcal{S}(V_\ell(F)\oplus F^2)$ and $h\in \SL_2(F).$ We have
    \begin{align*}
        Z_\ell(r_\ell(h)f,s)&=\int_{N(F)\backslash \SL_2(F)} e^{H(gh^{-1})(2-\ell-s)}r_\ell(g)f(0_{\underline{2\ell}},0,1)d\dot{g}\\
        &= \int_K \int_{F^\times}|t|^se^{H(kh^{-1})(2-\ell-s)}r_\ell(k)f(0_{\underline{2\ell}},0,t)dtdk\\
        &=Z_\ell(f,s)+\int_K\int_{F^\times} |t|^s\left(e^{H(kh^{-1})(2-\ell-s)}-1\right)r_\ell(k)f(0_{\underline{2\ell}},0,t) dtdk.
    \end{align*}
    Note that the integrand is zero when $s=2-\ell.$ Therefore, the latter term is holomorphic at $s=2-\ell,$ and it vanishes at $s=2-\ell$ if $\zeta(s)$ is holomorphic at $s=2-\ell$. This implies $\tilde{c}_\ell(f)=\tilde{c}_\ell(r_\ell(h)f).$
\end{proof}

Define another $\SL_2(F)$-equivariant map
\begin{align*}
    \tilde{d}_\ell:(r_\ell,\mathcal{S}(V_{\ell}(F)\oplus F^2))&\lto (r_{\ell-1},\mathcal{S}(V_{\ell-1}(F)\oplus F^2))\\
f&\longmapsto \mathcal{F}_2\big(v\mapsto f(v,0,0)\big).
\end{align*}
It  descends to a linear operator
\begin{align*}
    \tilde{d}_\ell:\widetilde{\mathcal{S}}(X_{\ell}(F))\lto \widetilde{\mathcal{S}}(X_{\ell-1}(F)).
\end{align*}

\begin{lem}\label{lem:ddescent}
    Let $\ell\ge 1$. For $f\in \widetilde{\mathcal{S}}(X_{\ell}(F))$ and $x\in X_{\ell-1}^\circ(F),$
    \begin{align*}
        I(\tilde{d}_\ell(f))(x)=\lim_{|a|\to 0} |a|^{\ell-2}\int_{F} I(f)(ax,0,ay) \psi(y)dy.
    \end{align*}
In particular, if $I(f)=0,$ then $I(\tilde{d}_\ell(f))=0$. 
\end{lem}

\begin{proof}
Let $f\in \mathcal{S}(V_{\ell+1}(F)).$ We may assume $f$ is $r_\ell(K)$-invariant. Then for $x\in X_{\ell-1}^\circ(F),$ 
\begin{align*}
    I(\tilde{d}_\ell(f))(x)=\int_{F^\times}\left(\int_{F} |t|^{\ell-2} f(tx,0,ty,0,0)\psi(y)dy\right) d^\times t.
\end{align*}
For $a\in F$ consider the function
\begin{align*}
    f_{x,a}(t):=\int_{F} |t|^{\ell-2} f(tx,0,ty,0,t^{-1}a)\psi(y)dy.
\end{align*}
It is continuous in $a,$ so $$ I(\tilde{d}_\ell(f))(x)=\int_{F^\times}\lim_{|a|\to 0} f_{x,a}(t)d^\times t.$$ We claim the limit and the integral can be interchanged. Assuming the claim, we have
\begin{align*}
    I(\tilde{d}_\ell(f))(x)&=\lim_{|a|\to 0}\int_{F^\times}f_{x,a}(t)d^\times t =\lim_{|a|\to 0}\int_{F^\times}\int_{F} |t|^{\ell-2} f(tx,0,ty,0,t^{-1}a)\psi(y)dyd^\times t.
\end{align*}
For $a\neq 0,$ since $f$ is Schwartz, the integral over $y$ and $t$ are absolutely convergent. Thus by the Fubini-Tonelli theorem and a change of variables $t\mapsto ta,$ it equals
\begin{align*}
    &\lim_{|a|\to 0}\int_{F}\int_{F^\times} |t|^{\ell-2} f(tx,0,ty,0,t^{-1}a)\psi(y)d^\times tdy=\lim_{|a|\to 0}\int_{F} |a|^{\ell-2}I(f)(ax,0,ay) \psi(y)dy.
\end{align*}

To justify the claim, observe that
\begin{align*}
    |f_{x,a}(t)|\le |t|^{\ell-3}f'(tx,t^{-1})
\end{align*}
for some $f'\in \mathcal{S}(F^{2\ell-1})$ depending only on $f$. Since $t\mapsto |t|^{\ell-3}f'(tx,t^{-1})$ is a function in $L^1(F^\times, d^\times t),$ the claim follows from the Lebesgue's dominated convergence theorem.  
\end{proof}

Define the following subspaces of $\widetilde{\mathcal{S}}(X_\ell(F))$
\begin{align*}
    W_\ell&=W_\ell(F):=\left(\mathcal{S}(V_\ell(F))\hat{\otimes} \mathcal{S}(F^2-\{0\})\right)_{r_\ell(\SL_2(F))},\\
    W_\ell'&=W_\ell'(F):=I^{-1}(\mathcal{S}_{\mathrm{ES}}(X_\ell(F))).
\end{align*}
Here $\hat{\otimes}$ is the algebraic tensor product when $F$ is nonarchimedean, and is the completed projective tensor product when $F$ is archimedean. Clearly $I(W_\ell)=\mathcal{S}_{\mathrm{ES}}(X_\ell(F))$ and $W_\ell'=I^{-1}I(W_\ell).$ Since $\tilde{d}_\ell(W_\ell)=0,$ from Lemma \ref{lem:ddescent} we deduce that

\begin{cor}\label{eq:dvanish}
    We have $I(\tilde{d}_\ell(W'_\ell))=0$.\qed
\end{cor}

\subsection{Poisson summation formulae}

Let $E$ be a number field and $\psi=\otimes \psi_v:E\backslash \A_E\to \CC^\times$ be a nontrivial additive character. Recall our convention of adelic Schwartz spaces in \S \ref{ssec:Schwartz}. Let $\mathcal{S}(V_{\ell}(\A_E)\oplus \A_E^2)$ be defined with respect to the basic functions $\one_{V_{\ell+1}(\OO_v)}.$ Define $\widetilde{\mathcal{S}}(X_{\ell}(\A_E))$ to be its $r_\ell(\SL_2(\A_E))$-coinvariant space. We have a Fourier transform
\begin{align*}
    \widetilde{\mathcal{F}}_{X_\ell}:=\otimes_v \widetilde{\mathcal{F}}_{X_\ell,v}:\widetilde{\mathcal{S}}(X_{\ell}(\A_E))\longrightarrow\widetilde{\mathcal{S}}(X_{\ell}(\A_E)).
\end{align*}
One defines analogously $Z_{\ell}(f,s)$ for $f\in \mathcal{S}(V_{\ell}(\A_E)\oplus \A_E^2).$ Let
\begin{align*}
    \tilde{c}_\ell:\mathcal{S}(V(\A_E)\oplus \A_E^2)&\lto \CC\\
    f&\lmto\begin{cases}
        Z_{\ell}(f,2-\ell) & \textrm{if } \ell>2,\\
        \frac{d}{ds}sZ_{\ell}(f,s+2-\ell)\bigg|_{s=0} & \textrm{if } \ell={1,2}.
    \end{cases}
\end{align*}
Define linear maps
\begin{align*}
    \tilde{d}_\ell:\mathcal{S}(V_\ell(\A_E)\oplus \A_E^2)\lto \mathcal{S}(V_{\ell-1}(\A_E)\oplus \A_E^2)
\end{align*}
on pure tensors by $\tilde{d}_\ell(\otimes f_v):=\otimes \tilde{d}_{\ell,v}(f_v)$. The maps $\tilde{d}_\ell$ are clearly $r_\ell(\SL_2(\A_E))$-invariant. For $i< \ell,$ let
\begin{align*}
    \tilde{d}_{\ell,i}:=\tilde{d}_{i+1}\circ \cdots\circ \tilde{d}_{\ell-1}\circ \tilde{d}_\ell
\end{align*}
These are the linear maps defined in \cite{Getz:Quad}; the notations therein are $c_\ell, d_\ell,d_{\ell,i}$. By \cite[Theorem 7.1]{Getz:Quad} $\tilde{c}_\ell$ for $\ell\ge 3$ are $r_\ell(\SL_2(\A_E))$-invariant, and $\tilde{c}_2+\tilde{c}_1\circ \tilde{d}_2$ is $r_2(\SL_2(\A_E))$-invariant. 
 
The main result of \cite{Getz:Quad} is the following summation formulae.

\begin{thm}\label{thm:Poisson:lift}
Suppose $\ell\ge 2$. Let $f\in \widetilde{\mathcal{S}}(X_\ell(\A_E)).$ Then
\begin{align*}
    &\sum_{\xi\in X_\ell^\circ(E)}I(f)(\xi)+\tilde{c}_\ell(f)+\sum_{i=1}^{\ell-1} \bigg(\tilde{c}_i(\tilde{d}_{\ell,i}(f))+\sum_{\xi\in X_i^\circ(E)} I(\tilde{d}_{\ell,i}(f))(\xi)\bigg)\\
    &+I(\tilde{d}_{\ell,0}(f))(0)+|D|^{\frac{1}{2}}\zeta(2)\tilde{d}_{\ell,0}(f)(0)\\
&=\sum_{\xi\in X_\ell^\circ(E)}I(\widetilde{\mathcal{F}}_{X_\ell}(f))(\xi)+\tilde{c}_\ell(\widetilde{\mathcal{F}}_{X_\ell}(f))+\sum_{i=1}^{\ell-1} \bigg(\tilde{c}_i(\tilde{d}_{\ell,i}(\widetilde{\mathcal{F}}_{X_\ell}(f)))+\sum_{\xi\in X_i^\circ(E)} I(\tilde{d}_{\ell,i}(\widetilde{\mathcal{F}}_{X_\ell}(f)))(\xi)\bigg)\\
    &+I(\tilde{d}_{\ell,0}(\widetilde{\mathcal{F}}_{X_\ell}(f)))(0)+|D|^{\frac{1}{2}}\zeta(2)\tilde{d}_{\ell,0}(\widetilde{\mathcal{F}}_{X_\ell}(f))(0).
\end{align*}
\qed
\end{thm}

\begin{rem}
    We point out that there is a typo in \cite[Theorem 1.1]{Getz:Quad}. Based on \cite[Theorem 4.1]{Getz:Quad}, in the notation of loc. cit., terms $I(d_{\ell,0}(f))(0)$ and $I(d_{\ell,0}(\mathcal{F}_{X_\ell}(f)))(0)$ are missed.
\end{rem}

We implement the idea in \cite[\S 12]{Getz:Hsu} to prove that the local Fourier transform $\widetilde{\mathcal{F}}_{X_\ell}$ descends to an automorphism of $I(\widetilde{\mathcal{S}}(X_\ell(F)))$ as a consequence of the summation formula.

\begin{lem} \label{lem:nonzero}
Let $\ell\ge 2$. Suppose $F$ is nonarchimedean. For any $x_0 \in X_{\ell}^\circ (F)$ there exists $f \in W_{\ell}$ such that $I(f)(0)=0$ and $I(\widetilde{\mathcal{F}}_{X_\ell}(f))(x_0) \neq 0.$
\end{lem}

\begin{proof}  Consider
\begin{align*}
    W:=\mathcal{S}(V_\ell(F)-\{0\})\otimes \mathcal{S}(F^2-\{0\}).
\end{align*}
Note that $I(f)(0)=0$ for $f\in W.$ The group $\mathrm{O}_{V_\ell}(F)$ stabilizes $W$ and acts transitively on $X_\ell^\circ(F).$ Therefore, if $I(\widetilde{\mathcal{F}}_{X_\ell}(f))(x_0)=0$ for all $f\in W,$ then $I(\widetilde{\mathcal{F}}_{X_\ell}(f))=0$ for all $f\in W$ by Proposition \ref{prop:action}. However, $\mathcal{F}_{\wedge}(\mathcal{S}(F^2-\{0\}))$ has codimension $1$ in $\mathcal{S}(F^2),$ so one can easily find $f\in W$ such that $I(\widetilde{\mathcal{F}}_{X_\ell}(f))\neq 0.$
\end{proof}

\begin{lem}\label{lem:des}
Let $\ell\ge 2.$ There is a (unique) linear operator $\mathcal{F}:I(\widetilde{\mathcal{S}}(X_{\ell}(F)))\to I(\widetilde{\mathcal{S}}(X_{\ell}(F)))$ which makes the following diagram commute
$$
\begin{CD}
\widetilde{\mathcal{S}}(X_\ell(F)) @>{\widetilde{\mathcal{F}}_{X_\ell}}>> \widetilde{\mathcal{S}}(X_\ell(F))\\
@V{ I}VV @V{ I}VV
\\
I(\widetilde{\mathcal{S}}(X_{\ell}(F))) @>{\mathcal{F}}>>  I(\widetilde{\mathcal{S}}(X_{\ell}(F))).
\end{CD}
$$
\end{lem}
\begin{proof}
        Suppose $\widetilde{\mathcal{F}}_{X_\ell}(\ker I) \leq \ker I.$ Since $\widetilde{\mathcal{F}}_{X_\ell}$ has order $2,$ one has $\widetilde{\mathcal{F}}_{X_\ell}(\ker I)=\ker I$ and the lemma follows. 
  Thus it suffices to show if $f\in\mathcal{S}(V_{\ell+1}(F))$ satisfies $I(f)=0$ then $I(\widetilde{\mathcal{F}}_{X_\ell}(f))=0.$      
        
 Let $E$ be a number field such that $E_v=F$ for some place $v$. Let $f_v\in \mathcal{S}(V_{\ell+1}(F))$ such that $I(f_v)=0$. Choose two finite places $v_1, v_2$ distinct from $v$  and functions $f_{v_i} \in \mathcal{S}(V_{\ell+1}(E_{v_i}))$ such that $f_{v_1}\in W_\ell(E_{v_1}),$ $\widetilde{\mathcal{F}}_{X_\ell}(f_{v_2}) \in W_{\ell}(E_{v_2})$ and $$ I(f_{v_1})(0)=I(\widetilde{\mathcal{F}}_{X_\ell}(f_{v_2}))(0)=0.$$ Choose $f^{vv_1v_2} \in \mathcal{S}(V_{\ell+1}(\A_E^{vv_1v_2}))$. By the fact that $\tilde{d}_\ell(W_\ell(E_{v_i}))=0$ and Lemma \ref{lem:easy} below, terms in Theorem \ref{thm:Poisson:lift} involving $\tilde{c}_i, \tilde{d}_{\ell,i}$ vanish, and we have
    \begin{align}\label{eq:sumvanish}
        \sum_{\xi\in X_\ell^\circ(E)} I(\widetilde{\mathcal{F}}_{X_\ell}(f_{v}f_{v_1}f_{v_2}f^{vv_1v_2}))(\xi)=\sum_{\xi\in X_\ell^\circ(E)} I(f_{v}f_{v_1}f_{v_2}f^{vv_1v_2})(\xi)=0.
    \end{align}
    Let $\xi_0 \in X_\ell^\circ(E).$  By Lemma \ref{lem:nonzero} we can choose $f_{v_1}$ as above so that $I(\widetilde{\mathcal{F}}_{X_\ell}(f_{v_1}))(\xi_0)\neq 0.$  Since $X_\ell(E)$ is discrete in $X_\ell(\mathbb{A}_E)$, we may choose $f_{v_2}$ and $f^{vv_1v_2}$ such that \eqref{eq:sumvanish} equals to $I(\widetilde{\mathcal{F}}_{X_\ell}(f_v))(\xi_0)$. Thus $I(\widetilde{\mathcal{F}}_{X_\ell}(f_{v}))$ vanishes on $X_\ell^\circ(E).$ As $X_\ell^\circ(E)$ is dense in $X_\ell^\circ(F),$ the assertion follows by continuity. 
\end{proof}

\section{Agreement of Schwartz spaces: nonarchimedean}\label{sec:agree}

Proposition \ref{prop:action} and  Lemma \ref{lem:des} imply that for $\ell\ge 2$ the smooth $\mathrm{O}_{V_{\ell+1}}(F)$-representation $(\tilde{\sigma}_\ell,\widetilde{\mathcal{S}}(X_{\ell}(F)))$ equips $I(\widetilde{\mathcal{S}}({X_\ell}(F)))$ with a structure of a smooth $\mathrm{O}_{V_{\ell+1}}(F)$-representation $\sigma_\ell'$. By \eqref{eq:Pequal}, the actions of $\sigma_\ell(P_{\ell+1}(F))$ and $\sigma'_\ell(P_{\ell+1}(F))$ on $L^2(X_{P_{\ell}}(F))$ agree. 

In this and the following sections, we show that indeed $\sigma_\ell=\sigma_\ell'.$ More precisely, we show that for $\ell\ge 2$ the operator $\mathcal{F}$ in Lemma \ref{lem:des} agrees with $\mathcal{F}_{X_\ell}$ and that the asymptotic maps $(c_\ell,\varepsilon(\ell-2,\psi)d_\ell)$ defined in \S\ref{sec:BK} are the descent of $(\tilde{c}_\ell,\tilde{d}_\ell).$ 

In this section, $F$ is a nonarchimedean local field of characteristic $0.$ For part (ii) of the following lemma, let $\mathbb{P}X_2 \subset \mathbb{P}V_2$ be the projective scheme cut out by $Q_2.$

\begin{lem}\label{lem:easy}
\begin{enumerate}[label=(\roman*)]
    \item  Suppose $\ell\ge 3$. For $f\in W_\ell'$, we have $\tilde{c}_\ell(f)= I(f)(0).$
    \item Suppose $\ell=2$. For $f\in\mathcal{S}(V_2(F)\oplus F^2),$ $x\in X_{2}(F)^1\cong \mathbb{P}X_2(F)$ and  $n\in \ZZ$ sufficiently large (depending only on $f$),
    \begin{align*}
        I(f)(\varpi^nx)=n\tilde{c}_2(f)+\tilde{a}_2(f)(x)
    \end{align*}
    for some function $\tilde{a}_2(f)\in C^\infty(\mathbb{P}X_2(F))$, and the difference
    \begin{align*}
 Z_{2}(f,s)-\tilde{c}_2(f)\zeta(s)
    \end{align*}
    is entire in $s.$ Both $\tilde{c}_2(f)$ and $\tilde{a}_2(f)$ depend only on $I(f)$.
    \item For $f\in \widetilde{\mathcal{S}}(X_1(F)),$ the integral defining $I(f)(x)$ is absolutely convergent for all $x\in X_1(F)$. We have $\tilde{c}_1(f)=\zeta(1)^{-1}I(f)(0).$
\end{enumerate}

\end{lem}

\begin{proof}
Let $f\in \mathcal{S}(V_{\ell+1}(F))$. Note that $I$ and $Z_{\ell}$ are by definition $r_\ell(\SL_2(\OO))$-invariant, so we may assume $f$ is $r_\ell(\SL_2(\OO))$-invariant. By the Iwasawa decomposition, we have for $x\in X_\ell^\circ(F)$
\begin{align*}
        I(f)(x)&=\int_{F^\times} |t|^{2-\ell} f(t^{-1}x,0,t)d^\times t,
    \end{align*}
    and 
    \begin{align*}
        Z_\ell(f,s)=\int_{F^\times} f(0_{\underline{2\ell}},0,t)|t|^sd^\times t.
    \end{align*}
In the computation of the integrals above, $\psi$ only affects the value $\mathrm{vol}(\OO^\times)$ in the intermediate computation, but not the conclusion. Therefore, for simplicity we assume $\psi$ is unramified and thus $\mathrm{vol}(\OO^\times)=1$.

Choose $A\in \ZZ_{\ge 0}$ such that 
\begin{align*}
    f(v+w)=f(v),\quad &\forall\, v\in V_{\ell+1}(F), w\in \varpi^A V_{\ell+1}(\OO),\\ 
    f(v)=0, \quad &\forall v\notin \varpi^{-A}V_{\ell+1}(\OO).
\end{align*}
Let $x\in X_\ell(F)^1$. For any $n\ge 2A-1$ 
\begin{align}\label{eq:exp}
\begin{split}
    I(f)(\varpi^n x)&=\sum_{j=-A}^{\infty} q^{j(\ell-2)}f(\varpi^{n-j}x,0,\varpi^j)\\
    &=\sum_{j=-A}^{A-1} q^{j(\ell-2)}f(0_{\underline{2\ell}},0,\varpi^j)+\sum_{j=A}^{n-A} q^{j(\ell-2)}f(0)+\sum_{j=n-A+1}^{n+A} q^{j(\ell-2)}f(\varpi^{n-j}x,0,0).
\end{split}
\end{align}
For $\ell\neq 2,$ this is
\begin{align*}
&\sum_{j=-A}^{A-1} q^{j(\ell-2)}f(0_{\underline{2\ell}},0,\varpi^j)+\frac{q^{A(\ell-2)}}{1-q^{\ell-2}}f(0)-\frac{q^{(n-A+1)(\ell-2)}}{1-q^{\ell-2}}f(0)\\
    &+q^{(n-A+1)(\ell-2)}\sum_{j=0}^{2A-1} q^{j(\ell-2)}f(\varpi^{A-j-1}x,0,0).
\end{align*}
Suppose $\ell\ge 3$ and $f\in W_\ell'$. By taking $n$ large we have 
\begin{align*}
    I(f)(0)=I(f)(\varpi^n x)=\frac{q^{A(\ell-2)}}{1-q^{\ell-2}}f(0)+\sum_{j=-A}^{A-1} q^{j(\ell-2)}f(0_{\underline{2\ell}},0,\varpi^j).
\end{align*}
On the other hand, since $\zeta(s)$ is holomorphic at $s=2-\ell,$
\begin{align*} 
    \tilde{c}_\ell(f)=Z_{\ell}(f,2-\ell)&=\int_{F^\times}f(0_{\underline{2\ell}},0,a)|a|^s d^\times a \bigg |_{s=2-\ell}\\
    &=\int_{|a|\le q^{-A}}f(0)|a|^s d^\times a \bigg |_{s=2-\ell}+\sum_{j=-A}^{A-1} q^{j(\ell-2)}f(0_{\underline{2\ell}},0,\varpi^j)\\
    &=\frac{q^{A(\ell-2)}}{1-q^{\ell-2}}f(0)+\sum_{j=-A}^{A-1} q^{j(\ell-2)}f(0_{\underline{2\ell}},0,\varpi^j)\\
    &=I(f)(0).
\end{align*}
The same argument applies when $\ell=1$. This proves (i) and (iii).

Now assume $\ell=2$. Then 
\begin{align*}
    Z_2(f,s)=q^{-As}\zeta(s)f(0)+\sum_{j=-A}^{A-1} q^{-sj}f(0_{\underline{4}},0,\varpi^j),
\end{align*}
so $\tilde{c}_2(f)=f(0).$ We can write \eqref{eq:exp} as
\begin{align}\label{eq:forell=2}
\begin{split}
    I(f)(\varpi^nx)&=n\tilde{c}_2(f)+(-2A+1)f(0)+\sum_{j=-A}^{A-1} f(0_{\underline{4}},0,\varpi^j)+\sum_{j=-A+1}^A f(\varpi^{-j}x,0,0)\\
    &=:n\tilde{c}_2(f)+\tilde{a}_2(f)(x).
\end{split}
\end{align}
Since both $I(f)$ and $\tilde{c}_2(f)$ do not depend on the choice of $A$, $\tilde{a}_2(f)$ is well-defined. If $I(f)=0,$ then clearly both $\tilde{c}_2(f)$ and $\tilde{a}_2(f)$ are zero. This completes the proof. 
\end{proof}

\begin{prop}\label{prop:min:rep}
Suppose $\ell \ge 2$. Then $I(\widetilde{\mathcal{S}}(X_\ell(F)))=\mathcal{S}(X_\ell(F)),$ and the map $I$ is $\mathrm{O}_{V_{\ell+1}}(F)$-equivariant. In particular, the diagram
$$
\begin{CD}
\widetilde{\mathcal{S}}(X_\ell(F)) @>{\widetilde{\mathcal{F}}_{X_\ell}}>> \widetilde{\mathcal{S}}(X_\ell(F))\\
@V{ I}VV @V{ I}VV
\\
\mathcal{S}(X_\ell(F)) @>{\mathcal{F}_{X_\ell}}>>  \mathcal{S}(X_\ell(F))
\end{CD}
$$
commutes. Moreover, for $\ell\ge 3$  we have a commutative diagram of $\mathrm{O}_{V_\ell}(F)$-modules
    \begin{align*}
    \begin{CD}
    \widetilde{\mathcal{S}}(X_\ell(F)) @>{(\tilde{c}_\ell,\tilde{d}_{\ell})}>> \CC\oplus\widetilde{\mathcal{S}}(X_{\ell-1}(F))\\
    @V{I}VV @V{(\mathrm{id},I)}VV
    \\
    \mathcal{S}(X_\ell(F)) @>{(c_\ell,\varepsilon(\ell-2,\psi)d_{\ell})}>>  \CC\oplus\mathcal{S}(X_{\ell-1}(F)).
    \end{CD}
    \end{align*}
\end{prop}

\begin{proof}
For the first statement, we can assume $\psi$ is unramified. Let
\begin{align*}
     \mathcal{S}:=\mathcal{S}(X_\ell^\circ(F))+\mathcal{F}\left(\mathcal{S}(X_\ell^\circ(F))\right) <L^2(X_{\ell}(F)).
 \end{align*}
By the proof of \cite[Lemma 2.13]{GK:cone} $\mathcal{S}$ is a smooth $\mathrm{O}_{V_{\ell+1}}(F)$-subrepresentation of $\sigma'_\ell,$ and it is irreducible by Lemma \ref{lem:Iact}.
We claim $\mathcal{S}$ contains $b_\ell.$  Assume the claim for the moment.
Since $b_{\ell}=I(\one_{V_{\ell+1}(\OO)}),$ $b_\ell$ is spherical with respect to both $\sigma_\ell$ and $\sigma'_\ell.$    
Using the Iwasawa decomposition and the fact that 
unramified irreducible smooth representations are determined by spherical vectors, we deduce that $\mathcal{S}=\mathcal{S}(X_\ell(F))$ as $\mathrm{O}_{V_{\ell+1}}(F)$-representations. In particular, $\mathcal{S}_{\mathrm{ES}}(X_\ell(F))\le \mathcal{S}$ for $\ell\ge 3$ by Theorem \ref{thm:asymplge3}. This is also true for $\ell=2$ by looking at the asymptotics of $b_2$ (cf. \cite[Corollary 6.7]{Hsu:asymptotics}). It follows that 
\begin{align*}
\mathcal{S}=I(W_\ell)+\mathcal{F}(I(W_\ell))=I(W_\ell)+I(\widetilde{\mathcal{F}}_{X_{\ell}}(W_\ell))=I(\widetilde{\mathcal{S}}(X_{\ell}(F))).
\end{align*}
Therefore, the map $I$ is an $\mathrm{O}_{V_{\ell+1}}(F)$-equivariant surjection onto $\mathcal{S}(X_\ell(F))$. In particular, $\mathcal{F}=\mathcal{F}_{X_\ell}$. To see $b_\ell\in \mathcal{S}$, it suffices to show the existence of  functions $f_1,f_2\in \mathcal{S}(F^2)$ satisfying the following:
 \begin{align*}
     f_1+\mathcal{F}_\wedge(f_2)=\one_{\OO^2}, \textrm{ and }  f_i(0)=0, \,\,I(\one_{V_\ell(\OO)}\otimes f_i)(0)=0, \,\,\textrm{ for } i=1,2.
 \end{align*}
 Here $I(\one_{V_{\ell}(\OO)} \otimes f_i)(0)$ is defined because $f_i(0)=0.$
 One can find such a pair in the subspace $\langle \one_{\varpi^j\OO^2}: j \in \ZZ  \rangle\subset\mathcal{S}(F^2),$ for example. 
 
For the second statement, suppose $\ell\ge 3$ and $\psi$ is arbitrary. By Lemma \ref{lem:easy}(i) $\tilde{c}_\ell$ induces an $\mathrm{O}_{V_\ell}(F)$-equivariant linear functional 
\begin{align*}
    \tilde{c}_\ell':\widetilde{\mathcal{S}}(X_\ell(F))/I^{-1}(\mathcal{S}(X_\ell^\circ(F)))\tilde{\lto}\mathcal{S}(X_\ell(F))/\mathcal{S}(X_\ell^\circ(F))\lto \CC.
\end{align*}
Since $\mathcal{S}(X_\ell(F))/\mathcal{S}(X_\ell^\circ(F))\cong \CC\oplus \mathcal{S}(X_{\ell-1}(F))$ by Theorem \ref{thm:asymplge3}, and $\tilde{c}_\ell'$ and $c_\ell$ agree on $\CC$ by Lemma \ref{lem:easy}(i), we conclude $c_\ell\circ I=I\circ \tilde{c}_\ell$. 

We are left to show $\varepsilon(\ell-2,\psi)\, d_\ell\circ I=I\circ \tilde{d}_\ell$. By Theorem \ref{thm:asymplge3}
$d_{\ell}\circ I$ induces an $\mathrm{O}_{V_\ell}(F)$-equivariant isomorphism 
\begin{align*} 
 \iota:\widetilde{\mathcal{S}}(X_\ell(F))/W'_\ell   \tilde{\lto} \mathcal{S}(X_\ell(F))/\mathcal{S}_{\mathrm{ES}}(X_\ell(F))\tilde{\lto}\mathcal{S}(X_{\ell-1}(F)).
\end{align*}
Thus by Corollary \ref{eq:dvanish} we have an $\mathrm{O}_{V_\ell}(F)$-equivariant surjection
\begin{align*}
    J:\mathcal{S}(X_{\ell-1}(F))\xrightarrow{\iota^{-1}} \widetilde{\mathcal{S}}(X_\ell(F))/W'_\ell \xrightarrow{\,\,\tilde{d}_\ell\,\,}  \widetilde{\mathcal{S}}(X_{\ell-1}(F))/\tilde{d}_\ell(W'_\ell)\xrightarrow{\,\,I\,\,}  \mathcal{S}(X_{\ell-1}(F)).
\end{align*}
 As $\mathcal{S}(X_{\ell-1}(F))$ is an irreducible $\mathrm{O}_{V_\ell}(F)$-representation, we conclude that there is a nonzero constant $C_\psi$ such that $C_\psi d_\ell\circ I=I\circ \tilde{d}_\ell.$ When $\psi$ is unramified, $C_\psi=1$ because $J(b_{\ell-1})=b_{\ell-1}$ by our choice of $d_\ell$.

For general $\psi,$ let $f\in  \mathcal{S}(V_{\ell+1}(F))$ be invariant under $r_\ell(\SL_2(\mathcal{O})).$ Let us compute
\begin{align*}
    \tilde{c}_{\ell-1}\circ \tilde{d_\ell}(f)=c_{\ell-1}\circ I\circ \tilde{d}_\ell(f)=C_\psi c_{\ell-1}\circ d_\ell\circ I(f).
\end{align*}
By \eqref{eq:Tate}
\begin{align*}
    Z_{\ell-1}(\tilde{d_\ell}(f),s)&=\int_{F^\times} |t|^s\mathcal{F}_2(f)(\underline{0}_{2\ell-2},0,t,0,0) d^\times t\\
    &=\frac{\epsilon(1-s,\psi)\zeta(s)}{\zeta(1-s)}\int_{F^\times} |t|^{1-s} f(\underline{0}_{2\ell-2},0,t,0,0)d^\times t.
\end{align*}
For $\mathrm{Re}(s)<1,$ by Lebesgue's dominated convergence theorem  
\begin{align*}
    \int_{F^\times} |t|^{1-s} f(\underline{0}_{2\ell-2},0,t,0,0)d^\times t&=\lim_{|a|\to 0}\int_{F^\times} |t|^{1-s} f(\underline{0}_{2\ell-2},0,t,0,t^{-1}a)d^\times t\\
    &=\lim_{|a|\to 0} |a|^{1-s}\int_{F^\times} |t|^{1-s} f(\underline{0}_{2\ell-2},0,ta,0,t^{-1})d^\times t.
\end{align*}
Therefore,
\begin{align*}
    \tilde{c}_{\ell-1}\circ \tilde{d_\ell}(f)=\varepsilon(\ell-2,\psi)\lim_{|a|\to 0} |a|^{\ell-2}I(f)(\underline{0}_{2\ell-2},0,a)\begin{cases}
        \frac{\zeta(3-\ell)}{\zeta(\ell-2)}  & \textrm{if } \ell>3,\\
        \zeta(1)^{-1} & \textrm{if }\ell=3.
    \end{cases}
\end{align*}
By comparing terms, we have $C_\psi=\varepsilon(\ell-2,\psi).$ 
\end{proof}

For $\ell=1$, $X_1=\A^1\cup \A^1$ is not normal, and hence not spherical. In view of the results above, we define
\begin{align*}
    \mathcal{S}(X_{1}(F)):=I(\widetilde{\mathcal{S}}(X_1(F))).
\end{align*} 
  Then the map $I:\widetilde{\mathcal{S}}(X_1(F))\longrightarrow \mathcal{S}(X_{1}(F))$ is by definition a $P_{2}(F)$-equivariant surjection.

Let $\CC_{1}$ be the representation $|\cdot|$ of $F^\times$.

\begin{prop}\label{prop:case1} 
  We have an exact sequence of smooth $\mathrm{O}_{V_1}(F)$-representations
    \begin{align*}
        0\longrightarrow \mathcal{S}(X_{1}^\circ(F))\longrightarrow \mathcal{S}(X_{1}(F))\xrightarrow{(c_1,(d_1', d_1))}  \mathrm{triv}\oplus\CC^2\longrightarrow 0,
    \end{align*}
    where $\CC^2$ is the irreducible representation $\mathrm{Ind}_{\mathrm{SO}_{V_1}}^{\mathrm{O}_{V_1}} \CC_1$ given by
    \begin{align*}
        \begin{psmatrix}
        a &\\
          & a^{-1}
        \end{psmatrix}.(v,w)=(|a|v,|a|^{-1}w), \quad \begin{psmatrix}
         & 1\\
        1  & 
        \end{psmatrix}.(v,w)=(w,v).
    \end{align*}
    Here $(d_1',d_1), c_1$ are normalized so that
    \begin{align*}
        (d_1',d_1)(b_1)=(1,1), \quad c_1(b_1)=1=\frac{b_1(0)}{\zeta(1)}.
    \end{align*}
    Furthermore, we have commutative diagrams of $\mathrm{SO}_{V_1}(F)$-modules
    \begin{align*}
    \begin{minipage}{0.48\textwidth}
    \vspace{0pt}
    \[
    \begin{CD}
    \widetilde{\mathcal{S}}(X_1(F)) @>{\quad\quad(\tilde{c}_1,\tilde{d}_{1})\quad\quad}>> \CC\oplus \widetilde{\mathcal{S}}(X_0(F))\\
    @V{I}VV @V{(\mathrm{id},I)}VV
    \\
    \mathcal{S}(X_{1}(F)) @>{(c_1,\varepsilon(-1,\psi)\zeta(2)d_{1})}>>  \CC\oplus \CC_{1}
    \end{CD}
\]
\end{minipage}
\hfill
\begin{minipage}{0.48\textwidth}
\vspace{6pt}
\[
    \begin{CD}
    \widetilde{\mathcal{S}}(X_1(F)) @>{\quad\quad \tilde{d}_{1}\quad\quad}>> \widetilde{\mathcal{S}}(X_0(F))\\
    @V{I}VV @V{\mathrm{ev}}VV
    \\
    \mathcal{S}(X_{1}(F)) @>{\quad\quad \varepsilon(-1/2,\psi)d_{1}'\quad\quad}>>  \CC_{1}.
    \end{CD}
    \]
\end{minipage}
    \end{align*}
    
    \medskip
    \noindent Here $\mathrm{ev}$ is the evaluation map of functions on $F^2$ at $(0,0).$
\end{prop}

\begin{proof}
For the first assertion, we may assume $\psi$ is unramified. Assume $f\in \mathcal{S}(V_2(F))$ is $r_1(\SL_2(\OO))$-invariant. Let $x\in X_1(F)^1.$ Using the notation in the proof of Lemma \ref{lem:easy}, for $n$ large by \eqref{eq:exp}
\begin{align}\label{eq:2comp}
\begin{split}
    I(f)(\varpi^n x)=&I(f)(0)-q^{-n}\Big(\frac{q^{A-1}}{1-q^{-1}}f(0)-q^{A-1}\sum_{j=0}^{2A-1} q^{-j}f(\varpi^{A-j-1}x,0,0)\Big).
\end{split}
\end{align}
Since $f$ is $r_1(\SL_2(\OO))$-invariant, by \eqref{eq:2comp} we have
\begin{align*}
    I(f)(\varpi^nx)=\begin{cases}
    I(f)(\varpi^n(1,0)) &\textrm{ if } x=(a,0), a\in \OO^\times,\\
    I(f)(\varpi^n(0,1))& \textrm{ if } x=(0,a), a\in \OO^\times.
    \end{cases}
\end{align*}
Define
\begin{align*}
    (d_1',d_1)(I(f))&:=\frac{q^n}{\zeta(-1)}\big(I(f)(\varpi^n(1,0))-I(f)(0), I(f)(\varpi^n(0,1))-I(f)(0)\big).
\end{align*}
This is well-defined by \eqref{eq:2comp}. We define $c_1(I(f)):=\frac{I(f)(0)}{\zeta(1)}$ so $c_1(I(f))=\tilde{c}_1(f)$ by Lemma \ref{lem:easy}(iii). The first assertion follows by definition.

For the second assertion, by Corollary \ref{eq:dvanish} we have an $\mathrm{SO}_{V_1}(F)$-equivariant surjection
\begin{align*}
    J:\CC^2 \tilde{\lto} \mathcal{S}(X_{1}(F))/\mathcal{S}_{\mathrm{ES}}(X_{1}(F))\tilde{\lto} \widetilde{\mathcal{S}}(X_1(F))/W_1'\xrightarrow{\tilde{d}_1} \widetilde{\mathcal{S}}(X_0(F))/\tilde{d}_1(W_1')\xrightarrow{\,\,I\,\,} \CC.
\end{align*}
Since $\tilde{d}_1\circ \tilde{\sigma}_1\begin{psmatrix}
a & \\
 & a^{-1}
\end{psmatrix}=|a|^{-1}\tilde{d}_1,$ we have $I\circ \tilde{d}_1=C_\psi d_1\circ I$ for some nonzero constant $C_\psi$. When $\psi$ is unramified, we have $\zeta(2)d_1(b_1)=\zeta(2)=I(\tilde{d}_1(\one_{V_2(\OO)})),$ so $C_\psi=\zeta(2)$. By Lemma \ref{lem:ddescent}, for $f\in \widetilde{\mathcal{S}}(X_1(F))$
\begin{align*}
    I(\tilde{d}_1(f))=I(\tilde{d}_1(f))(0)&=\lim_{|a|\to 0}\int_{F} |a|^{-1}I(f)(0,ay)\psi(y)dy.
\end{align*}
Let $z\in F^\times$ such that $\psi_u(y):=\psi(zy)$ is unramified. Let $d_u y$ be the self-dual Haar measure on $F$ with respect to $\psi_u$. By changing variables $y\mapsto zy$ and $a\mapsto az^{-1}$, we have
\begin{align*}
    I(\tilde{d}_1(f))=|z|^{3/2}\lim_{|a|\to 0}\int_{F} |a|^{-1}I(f)(0,ay)\psi_u(y)d_u y.
\end{align*}
Note that $|z|^{3/2}=\varepsilon(-1,\psi)$. By comparing terms $C_\psi=\varepsilon(-1,\psi)\zeta(2).$

Finally, for $f\in \widetilde{\mathcal{S}}(X_1(F))$ by Fourier inversion
\begin{align*}
    \tilde{d}_1(f)(0)&=\int_{F} f(0,y,0,0)dy\\
    &=\int_{F^2}\left(\int_{F} f(t,y,0,0)\psi(xt)\right)dxdy\\
    &=\int_{F^2}\left(\int_{F} \left(\tilde{\sigma}_1\begin{psmatrix}
        0 & 1\\
         1 & 0
    \end{psmatrix}f\right)(y,t,0,0)\psi(xt)\right)dxdy\\
    &=\int_{F^2}\tilde{d}_1\circ \tilde{\sigma}_1\begin{psmatrix}
        0 & 1\\
         1 & 0
    \end{psmatrix}(f)(y,x)dxdy
\end{align*}
By our choice of Haar measures, this is
\begin{align*}
    &\frac{1}{\zeta(2)\varepsilon(0,\psi)}\int_{N(F)\backslash \SL_2(F)}\tilde{d}_1\circ \tilde{\sigma}_1\begin{psmatrix}
        0 & 1\\
         1 & 0
    \end{psmatrix}(f)((0,1)g)d\dot{g}\\
    &=\frac{1}{\zeta(2)\varepsilon(0,\psi)}I\circ\tilde{d}_1\circ \tilde{\sigma}_1\begin{psmatrix}
        0 & 1\\
         1 & 0
    \end{psmatrix}(f)\\
    &=\frac{\varepsilon(-1,\psi)}{\varepsilon(0,\psi)}d_{1}\circ I\circ \tilde{\sigma}_1\begin{psmatrix}
        0 & 1\\
         1 & 0
    \end{psmatrix}(f)\\
    &=\varepsilon(-1/2,\psi)d_{1}'\circ I(f).
\end{align*}

\end{proof}

% From the proposition above and an analogous proof of Lemma \ref{lem:Iact} for $\ell=1,$ we conclude 
% \begin{cor}\label{cor:1}
%     We have
%     $$\mathcal{S}(X_{1}(F))=\mathcal{S}_{\mathrm{ES}}(X_{1}(F))+\mathcal{F}_{X_{1}}(\mathcal{S}_{\mathrm{ES}}(X_{1}(F))).$$
%     In particular, $\mathcal{S}_{\mathrm{ES}}(X_{1}(F))$ is of length $2$ as a $P_2(F)$-representation, and  $\mathcal{S}(X_{1}(F))$ is an $\mathrm{O}_{V_2}(F)$-representation of length at most $2$. \qed 
% \end{cor}

The scheme $X_2^\circ$ can be realized as the scheme consisting of $2$ by $2$ matrices of rank $1$ by
\begin{align*}
   (v_1,v_2,v_3,v_4)\longmapsto \begin{pmatrix}
        v_1 & v_3\\
        -v_4 & v_2
    \end{pmatrix}.
\end{align*}
Therefore, we can view $X_{2}$ as an affine $\SL_2^2$-spherical variety under the degree $2$ isogeny $\SL_2^2\to \mathrm{SO}_{V_2}.$ We record the following well-known statement. 

\begin{lem}\label{lem:sl2} Let $1_s$ denote the quasicharacter on $T(F)$ such that $1_{s}\left(\begin{psmatrix}
a &\\
 & a^{-1}
\end{psmatrix}\right)=|a|^{s}$. Let $\mathrm{Ind}_{B}^{\SL_2}(1_{s})$ be the normalized induction, and
\begin{align*}
    M_{w_0}(s):\mathrm{Ind}_{B}^{\SL_2}(1_{s})&\to \mathrm{Ind}_{B}^{\SL_2}(1_{-s})\\
    f&\mapsto \left(g\mapsto \int_{N(F)} f(\begin{psmatrix}
 & 1\\
-1 & 
\end{psmatrix}ug)\right)du 
\end{align*}
be the (unnormalized) intertwining operator. On $\mathrm{Re}(s)<0$, the operator $M_{w_0}(s)$ is holomorphic and is an isomorphism except at $1_{-1}$ where it has a nontrivial kernel. The module $\mathrm{Ind}_{B}^{\SL_2}(1_{-1})$ is of length $2.$ Its maximal semisimple subrepresentation (resp. quotient) is $\CC$ (resp. the Steinberg representation $\mathrm{St}$), and $\CC$ is the kernel of $M_{w_0}(1_{-1})$.
\qed
\end{lem}

 It follows that $\mathrm{Ind}_{B}^{\SL_2}(1_{-1})\otimes \mathrm{Ind}_{B}^{\SL_2}(1_{-1})$ is an $\SL_2^2(F)$-representation of length $4$ whose composition factors are the external tensor products $\CC\otimes \CC, \mathrm{St}\otimes \CC, \CC\otimes \mathrm{St}, \mathrm{St}\otimes \mathrm{St}$. 
 By conjugation we can assume $B^2$ is mapped onto $P_2\cap \mathrm{SO}_{V_2}$ under the isogeny $\mathrm{SL}_2^2 \to \mathrm{SO}_{V_2}.$  All of the composition factors above descend to representations of  $\mathrm{SO}_{V_2}(F)$.  The direct sum $(\mathrm{St}\otimes \CC) \oplus (\CC\otimes \mathrm{St})$  
may be identified with the $\mathrm{O}_{V_2}(F)$-representation 
\begin{align}
\sigma:=\mathrm{Ind}_{\mathrm{SO}_{V_2}}^{\mathrm{O}_{V_2}}(\mathrm{St} \otimes \CC).
\end{align}
Concretely, 
\begin{align*}
        \begin{psmatrix}
          & 1 & &\\
         1&  &  &\\
         & & 1 & \\
         & &  & 1
        \end{psmatrix}\in \mathrm{O}_{V_2}(F)
    \end{align*}
acts by permuting the two factors of $(\mathrm{St}\otimes \CC) \oplus (\CC\otimes \mathrm{St}).$

\begin{prop}\label{prop:case2}
\begin{enumerate}
\item[(i)]
We have a natural exact sequence of smooth $\mathrm{O}_{V_2}(F)$-representations
    \begin{align*}
        &0\longrightarrow \mathcal{S}_{\mathrm{ES}}(X_{2}(F))\longrightarrow \mathcal{S}(X_{2}(F))\xrightarrow{\,\,d_2\,\,} \pi\longrightarrow 0,
    \end{align*}
    where $\pi$ is of length $2$ and admits a nonsplit exact sequence of $\mathrm{O}_{V_2}(F)$-representations
    \begin{align*}
        0\longrightarrow \sigma \longrightarrow\pi \longrightarrow \CC\longrightarrow 0.
    \end{align*}
    Here the map $\pi\to \CC$ is given by taking the highest order term in the germ expansion.
      Let 
    \begin{align*}
      c_2:  \mathcal{S}(X_{2}(F))\xrightarrow{\,\,d_2\,\,} \pi\lto \CC.
    \end{align*}
    Then $c_2(I(f))=\tilde{c}_2(f)$. 
    
   \item[(ii)] We have an isomorphism $J:\pi\cong \mathcal{S}(X_{1}(F))$ of $P_{2}(F)$-representations such that \sloppy $J(d_2(b_2))=b_1$. Under this identification we have a commutative diagram of  $P_{2}(F)$-representations: $$
    \begin{CD}
    \widetilde{\mathcal{S}}(X_2(F)) @>{\tilde{d}_{2}}>> \widetilde{\mathcal{S}}(X_{1}(F))\\
    @V{I}VV @V{I}VV
    \\
    \mathcal{S}(X_{2}(F)) @>{\varepsilon(0,\psi)d_{2}}>>  \mathcal{S}(X_{1}(F)).
    \end{CD}
    $$
    \end{enumerate}
\end{prop}

\begin{proof}
By Proposition \ref{prop:min:rep}, $\mathcal{S}(X_2(F))=I(\widetilde{\mathcal{S}}(X_2(F)).$ Taking germs at the origin of functions in $\mathcal{S}(X_{2}(F))$ gives rise to an exact sequence of smooth $\SL_2^2(F)$-representations
\begin{align}\label{eq:2exact}
    0\lto \mathcal{S}(X_{2}^\circ(F))\lto \mathcal{S}(X_{2}(F))&\longrightarrow \pi''\lto 0.
\end{align}
By Lemma \ref{lem:easy}(ii), the map to $\pi''$ is given explicitly by
$$
I(f) \longmapsto \tilde{a}_2(f)(\cdot)-\tilde{c}_2(f) \log_{q}(|\cdot|).
$$
Note that both $\tilde{a}_2(f)$ and $\tilde{c}_2(f)$ only depend on $I(f),$ so this map is well-defined. 

 Since $\tilde{a}_2(f) \in C^\infty(\mathbb{P}X_2(F))=\mathrm{Ind}_{B}^{\SL_2}(1_{-1})\otimes \mathrm{Ind}_{B}^{\SL_2}(1_{-1}), $ 
one has an exact sequence of smooth $\SL_2^2(F)$-representations
\begin{align} \label{no:split}
\begin{split}
    0\lto  \pi'\lto  \pi''&\lto \CC \lto 0
\end{split}
\end{align}
where $\pi'$ is a subrepresentation of $\mathrm{Ind}_{B}^{\SL_2}(1_{-1})\otimes \mathrm{Ind}_{B}^{\SL_2}(1_{-1}).$ 
Here the second map is determined by the requirement that $-\tilde{c}_2(f)\log_{q}(|\cdot|)$ is sent to $\tilde{c}_2(f).$ 

To see $\CC\otimes \CC, \mathrm{St}\otimes \CC, \CC\otimes \mathrm{St}$ are composition factors of $\pi'$, by symmetry and Lemma \ref{lem:sl2} it suffices to show there is $f\in \mathcal{S}(V_{3}(F))$ such that $\tilde{a}_2(f)$ is nonconstant and $\tilde{c}_2(f)=0.$ We may assume $\psi$ is unramified. Consider $\tilde{b}_2':= \one_{\varpi \OO\times \varpi^{-1}\OO}\otimes \one_{\OO^4}$. It is $r_{2}(\SL_2(\OO))$-invariant. Then by \eqref{eq:forell=2} $\tilde{a}_2(\one_{V_3(\OO)}-\tilde{b}_2')$ is nonconstant and $\tilde{c}_2(\one_{V_3(\OO)}-\tilde{b}_2')=1-1=0$. 

Let $\pi:=\pi''/(\CC\otimes \CC)$. 
To complete the proof of (i) we are left with showing that $\pi'$ is a proper subrepresentation of $\mathrm{Ind}_{B}^{\SL_2}(1_{-1})\otimes \mathrm{Ind}_{B}^{\SL_2}(1_{-1})$. For $f\in \mathcal{S}(X_{2}(F))$, let
\begin{align*}
    f_s(x):&=\int_{F^\times}|a|^{s+1}f(ax) d^\times a,\quad x\in X_{2}^\circ(F)
\end{align*}
be the Mellin transform of $f$ along the quasicharacter $1_{s}\otimes 1_{s}$  (extended by meromorphic continuation). We claim that $f\in \mathcal{S}(X_{2}(F))$ only if $M_{w_0}(s)\otimes M_{w_0}(s)(f_s)$ is holomorphic at $s=-1$. Let $\tilde{f}\in \widetilde{\mathcal{S}}(X_2(F))$ such that $I(\tilde{f})=f.$ For $x\in X_{2}(F)^1,$ the Laurent series expansion of $f_s(x)$ at $s=-1$ is
\begin{align*}
     \tilde{c}_2(\tilde{f})\zeta(s+1)^2+(\tilde{a}_2(\tilde{f})(x)-\tilde{c}_2(\tilde{f}))\zeta(s+1)+O(1).
\end{align*}
Note that $\tilde{a}_2(\one_{V_3(\OO)})=1$ is a constant function, and thus $M_{w_0}(s)\otimes M_{w_0}(s)((b_2)_s)$ is holomorphic at $s=-1$ by Lemma \ref{lem:sl2}. \sloppy By Lemma \ref{lem:Iact} and the Iwasawa decomposition $\pi''\cong \mathcal{S}(X_{2}(F))/\mathcal{S}(X_{2}^\circ(F))$ is generated by (the image of) $b_2$ as an $\mathrm{\SL}^2_{2}(F)$-representation, so the claim follows. Now if $\pi'$ is not proper, then by Lemma \ref{lem:sl2} there is $\tilde{f}\in \widetilde{\mathcal{S}}(X_{2}(F))$ such that $M_{w_0}(-1)\otimes M_{w_0}(-1)(\tilde{a}_2(\tilde{f}))$ is nonzero, and thus $M_{w_0}(s)\otimes M_{w_0}(s)(I(f)_s)$ has a simple pole at $s=-1$, which is a contradiction.
    
Finally, to prove (ii) by Corollary \ref{eq:dvanish} we have a $P_{2}(F)$-equivariant surjection
    \begin{align*}
        \tilde{J}:\pi\cong \mathcal{S}(X_{2}(F))/\mathcal{S}_{\mathrm{ES}}(X_{2}(F))\cong \widetilde{\mathcal{S}}(X_{2}(F))/W_2'\xrightarrow{\,\,\tilde{d}_2\,\,} \widetilde{\mathcal{S}}(X_{1}(F))/\tilde{d}_2(W_2')\xrightarrow{\,\, I\,\,} \mathcal{S}(X_{1}(F)).
    \end{align*}
     Since $\pi$ as a $P_2(F)$-representation is of length $3$ whose maximal semisimple quotients are finite-dimensional,  $\tilde{J}$ must be injective and hence an isomorphism.  Let $J:\pi\tilde{\longrightarrow} \mathcal{S}(X_1)$ be the unique isomorphism such that $J(d_2(b_2))=b_1$. When $\psi$ is unramified, $J=\tilde{J}.$ For general $\psi,$ by Lemma \ref{lem:ddescent} for $f\in \widetilde{\mathcal{S}}(X_{2}(F))$ and $x\in X_{1}^\circ(F),$
    \begin{align*}
        I(\tilde{d}_2(f))(x)=\lim_{|a|\to 0} \int_{F} I(f)(ax,0,ay) \psi(y)dy.
    \end{align*}
    Say $x=(0,x_1)$. Then we have
    \begin{align*}
        I(\tilde{d}_2(f))(0,x_1)&=\lim_{|a|\to 0} \int_{F} I(f)(0,a,0,ax_1^{-1}y) \psi(y)dy\\
        &=|x_1|\lim_{|a|\to 0} \int_{F} I(f)(0,a,0,ay) \psi(x_1y)dy.
    \end{align*}
     Since $I(\tilde{d}_2(f))\in \mathcal{S}(X_1(F)),$ by Proposition \ref{prop:case1} the above function in $x_1$ is a function in $\mathcal{S}(F)+|\cdot| \mathcal{S}(F)$.  Let $z\in F^\times$ such that $\psi_u(y)=\psi(zy)$ is unramified. Let $d_u y$ be the self-dual Haar measure on $F$ with respect to $\psi_u$. Then 
    \begin{align*}
       &\lim_{|x_1|\to 0}I(\tilde{d}_2(f))(0,x_1)\\
       &=|z|^{-1/2}\lim_{|x_1|\to 0}|x_1|\lim_{|a|\to 0} \int_{F} I(f)(0,a,0,ay) \psi_u(x_1z^{-1}y)d_uy\\
       &=|z|^{1/2}\lim_{|x_1|\to 0}|x_1|\lim_{|a|\to 0} \int_{F} I(f)(0,a,0,ay) \psi_u(x_1y)d_uy.
    \end{align*}
    By comparing terms, we have $\tilde{J}=|z|^{1/2}J=\varepsilon(0,\psi)J.$    
\end{proof}

\begin{cor}\label{cor:2eq1}
For $f\in \mathcal{S}(X_{2}(F))$, $c_2(f)=c_1(d_2(f))$.
\end{cor}

\begin{proof}
By the computation in the proposition above, $c_2$ is equal to $c_1\circ d_2$ up to a nonzero scalar. Thus, it suffices to check the identity for $f=b_2$. The statement is independent of $\psi,$ so we can assume $\psi$ is unramified. By Proposition \ref{prop:case1} and  Proposition \ref{prop:case2} we have 
\begin{align*}
c_2(b_2)=1=c_1(b_1)=c_1(d_2(b_2)).\end{align*}
% By Proposition \ref{prop:case2}, $b_2$ generates the $\mathrm{O}_{V_2}(F)$-representation $\mathcal{S}(X_2(F))/\mathcal{S}_{\mathrm{ES}}(X_2(F))$.vanish on $\mathcal{S}_{\mathrm{ES}}(X_2(F))$ and are $\mathrm{O}_{V_2}(F)$-invariant linear functionals,
\end{proof}

We would like to replace $\pi''$ by $\pi \oplus (\CC\otimes\CC)$ in \eqref{eq:2exact} 
in analogy with the case of $\ell>2$ in Proposition \ref{prop:min:rep}, where $\CC\otimes\CC$ consists of germs of functions at the origin that are constant. But this is not possible if we consider \eqref{eq:2exact}  as an exact sequence of $\mathrm{O}_{V_2}(F)$-modules since $\CC\otimes\CC$ is the maximal semisimple $\mathrm{O}_{V_2}(F)$-subrepresentation of $\pi''$ (and $\pi'$). \quash{Indeed,  and the quotient consists of functions that are constant multiplies of $\log_{1/q}(|\cdot|)$ near zero.} However, if we view \eqref{eq:2exact} as an exact sequence of $P_2(F)$-modules, we can rewrite it as
\begin{align} \label{a2}
    0\lto \mathcal{S}(X_{2}^\circ(F))\lto \mathcal{S}(X_{2}(F))\xrightarrow{(a_2, d_2)} (\CC\otimes\CC)\oplus \pi\lto 0.
\end{align}
The map $a_2$ is defined as follows. The stabilizer of $(0:0:0:1)\in \mathbb{P}X_2(F)$ in $\mathrm{O}_{V_2}(F)$ is $P_2(F)$. For $f\in \mathcal{S}(X_2(F)),$ choose $\tilde{f}\in \widetilde{\mathcal{S}}(X_2(F))$ such that $I(\tilde{f})=f$. The linear functional
\begin{align*}
    a_2(f):=\tilde{a}_2(\tilde{f})(0:0:0:1)
\end{align*}
is well-defined by Lemma \ref{lem:easy}(ii) and $P_2(F)$-invariant. This gives a splitting \sloppy $\pi'=\sigma\oplus (\CC\otimes\CC)$ as $P_2(F)$-representations. We remark that for $f\in \mathcal{S}_{\mathrm{ES}}(X_{2}(F)),$ $a_2(f)=f(0).$

\begin{lem}\label{lem:der}
Let $f\in \mathcal{S}(V_2(F)\oplus F^2)$. We have 
\begin{align*}
    &\frac{d}{ds} \left(-\frac{Z_{2}(f,s)}{\zeta(s)}+\frac{Z_{1}(\tilde{d}_2(f), s+1)}{\varepsilon(-s,\psi)\zeta(s+1)}\right)\Bigg|_{s=0}\\
    &=\big(c_2(I(f))-a_2(I(f))\big)\log q.
\end{align*}
\end{lem}

\begin{proof}
We may assume $f$ is $r_2(\SL_2(\OO))$-invariant. Using the computation in Lemma \ref{lem:easy} and the notation therein, we have
\begin{align}\label{eq:2asy}
\begin{split}
    \zeta(s)^{-1}Z_{2}(f,s)&=\mathrm{vol}(\mathcal{O}^\times)\left(q^{-As}f(0)+(1-q^{-s})\sum_{j=-A}^{A-1}q^{-js}f(0_{\underline{4}},0,\varpi^j)\right)\\
    &=\mathrm{vol}(\mathcal{O}^\times)f(0)-s\mathrm{vol}(\mathcal{O}^\times)(\log q)\big(Af(0)-\sum_{j=-A}^{A-1}f(0_{\underline{4}},0,\varpi^j)\big)+O_{f}(s^2)
\end{split}
\end{align}
as $s \to 0.$ On the other hand, by \eqref{eq:Tate}
\begin{align*}
    Z_1(\tilde{d}_2(f),s+1)=\varepsilon(-s,\psi)\frac{\zeta(1+s)}{\zeta(-s)}\int_{F^\times} f(0_{\underline{2}},0,t,0,0)|t|^{-s}d^\times t,
\end{align*}
so
\begin{align*}
    \frac{Z_{1}(\tilde{d}_2(f), s+1)}{\varepsilon(-s,\psi)\zeta(s+1)}&=(1-q^s)\int_{|t|> q^{-A}} f(\underline{0}_2,0,t,0,0)|t|^{-s} d^\times  t+\mathrm{vol}(\mathcal{O}^\times)f(0)q^{As}.
\end{align*}
The coefficient of $s$ in the Taylor series expansion of this function at $s=0$ is
\begin{align*}
     \mathrm{vol}(\mathcal{O}^\times)(\log q)Af(0)-\log q \int_{|t|>q^{-A}} f(0_{\underline{2}},0,t,0,0)d^\times t.
\end{align*}
Consequently,  by \eqref{eq:forell=2} and Proposition \ref{prop:case2}
\begin{align*}
    &\frac{d}{ds} \left(-\frac{Z_{2}(f,s)}{\zeta(s)}+\frac{Z_{1}(\tilde{d}_2(f), s+1)}{\varepsilon(-s,\psi)\zeta(s+1)}\right)\Bigg|_{s=0}\\
    &=\mathrm{vol}(\mathcal{O}^\times)(\log q)\big(Af(0)-\sum_{j=-A}^{A-1}f(0_{\underline{4}},0,\varpi^j)\big)\\
    &+\mathrm{vol}(\mathcal{O}^\times)(\log q)(Af(0)-\sum_{j=-A+1}^{A} f(0_{\underline{2}},0,\varpi^{-j},0,0))\\
    &=(\log q)(\tilde{c}_2(f)-\tilde{a}_2(f)(0:0:0:1))\\
    &=(\log q)(c_2(I(f))-a_2(I(f))).
\end{align*}
\end{proof}

\section{Agreement of Schwartz spaces: archimedean}\label{sec:agree:arch}

In this section, $F$ is archimedean. 

\begin{lem}\label{lem:easy:arch}
\begin{enumerate}[label=(\roman*)]
    \item  Suppose $\ell\ge 3$. For $f\in W_\ell'$ we have
    \begin{align*}
        \tilde{c}_\ell(f)=\begin{cases}
            I(f)(0) &\textrm{if } \ell \textrm{ is odd and }F=\RR,\\
            0 & \textrm{otherwise.}
        \end{cases}
    \end{align*}     
    \item Suppose $\ell=2$. For $f\in\mathcal{S}(V_2(F)\oplus F^2),$ $x\in X_{2}(F)^1\cong \mathbb{P}X_2(F)$ 
    \begin{align*} 
        I(f)(ax)=-\tilde{c}_2(f)\log|a|+\tilde{a}_2(f)(x)+o(1) \quad\quad \textrm{ as } |a|\to 0
    \end{align*}
    for some function $\tilde{a}_2(f)\in C^\infty(\mathbb{P}X_2(F))$, and the difference
    \begin{align*}
 Z_{2}(f,s)-\frac{1}{s}\tilde{c}_2(f)
    \end{align*}
    is holomorphic for $\mathrm{Re}(s)>-1/2$. Both $\tilde{c}_2(f)$ and $\tilde{a}_2(f)$ depend only on $I(f)$.
    \item For $f\in \widetilde{\mathcal{S}}(X_1(F)),$ the integral defining $I(f)(x)$ is absolutely convergent for all $x\in X_1(F)$. We have $\tilde{c}_1(f)=\zeta(1)^{-1}I(f)(0).$
\end{enumerate}

\end{lem}

\begin{proof}
    Let $f\in \mathcal{S}(V_{\ell+1}(F)).$ We may assume $f$ is $r_\ell(K)$-invariant. Therefore, for $x\in X_\ell^\circ(F)$
    \begin{align*}
        I(f)(x)&=\int_{F^\times} |t|^{2-\ell} f(t^{-1}x,0,t)d^\times t
    \end{align*}
    and 
    \begin{align*}
        Z_\ell(f,s)=\int_{F^\times} f(0_{\underline{2\ell}},0,t)|t|^sd^\times t.
    \end{align*}
    Both integral converges absolutely when $\ell=s=1,$ in which case we can take $x=0$. This proves (iii).
    
   Suppose $\ell\ge 3$. Since these two operators are continuous in $f$, to compute both terms we now assume $f(v,0,t)=f_1(v)f_2(t)$ for some $f_1\in \mathcal{S}(F^{2\ell})$ and  $f_2\in \mathcal{S}(F)$.  We will prove the lemma for $F=\RR.$ The case $F=\CC$ is analogous by identifying $\CC=\RR^2,$ so we leave it to the reader. 
   
    Choose $r>0$ and complex numbers $b_0,\ldots, b_{\ell-2}$ such that 
   \begin{align*}
       \bigg|f_2(t)-\sum_{n=0}^{\ell-2}b_nt^n\bigg|\le |t|^{\ell-2+1/4} \textrm{ for all } |t|\le r. 
   \end{align*}
   Then for $a\in F^\times$
   \begin{align*}
       I(f)(ax)&=\int_{|t|>r} |t|^{2-\ell}f_1(t^{-1}ax)f_2(t) d^\times t+\int_{|t|\le r} |t|^{2-\ell}f_1(t^{-1}ax)\left(f_2(t)-\sum_{n=0}^{\ell-2} b_nt^n\right) d^\times t\\
       &+\sum_{n=0}^{\ell-2}\int_{|t|\le r} b_nt^n|t|^{2-\ell} f_1(t^{-1}ax)d^\times t.
   \end{align*}
   The first two terms are absolutely convergent and bounded above by a constant independent of $a$ and $x$. For the last term, by changing variables we obtain
   \begin{align*}
       \sum_{n=0}^{\ell-2} b_n a^n|a|^{2-\ell}\int_{|t|\le |a|^{-1}r}t^n|t|^{2-\ell} f_1(t^{-1}x)d^\times t.
   \end{align*}
    Therefore, if $f\in W_\ell'$ so that $I(f)(0)=\lim_{|a|\to 0}I(f)(ax)$ exists, each summand above is zero except possibly for $n=\ell-2$. In this case, we arrive at
    \begin{align*}
        b_{\ell-2}\mathrm{sgn}^{\ell}(a)\int_{|t|\le |a|^{-1}r} \mathrm{sgn}^\ell(t)f_1(t^{-1} x)d^\times  t.
    \end{align*}
    This term must be zero if $\ell$ is odd. For $\ell$ even, it must be $f_1(0)=0.$ 
    
    When $\ell$ is odd, we have 
    \begin{align*}
        I(f)(0)&=f_1(0)\left(\int_{|t|>r} |t|^{2-\ell}f_2(t) d^\times t+\int_{|t|\le r} |t|^{2-\ell}\left(f_2(t)-b_{\ell-2}t^{\ell-2}\right) d^\times t\right).
    \end{align*}
    On the other hand, for any $\ell\ge 3$ we have
    \begin{align*}
        Z_\ell(f,s)=f_1(0)\left(\int_{|t|>r} f_2(t)|t|^{s}d^\times+\int_{|t|>r} (f_2(t)-b_{\ell-2}t^{\ell-2})|t|^sd^\times t+\int_{|t|<r} b_{\ell-2}t^{\ell-2}|t|^s d^\times t\right).
    \end{align*}
    This is zero when $\ell$ is even. When $\ell$ is odd, the last integral vanishes, so $I(f)(0)=Z_\ell(f,2-\ell).$ This proves (i). 
    
    Suppose $\ell=2$. Let $x\in X_2(F)^1 $ and $|a|<1$. Write
    \begin{align*}
        I(f)(ax)&=\int_{|t|>1} f(t^{-1}ax,0,t)d^\times t+ \int_{|t|\le 1} f(t^{-1}ax,0,t)-f(t^{-1}ax,0,0)d^\times t\\
        &+\int_{|t|\le |a|^{-1}} f(t^{-1}x,0,0)d^\times t.
    \end{align*}
    The last term can be written as
    \begin{align*}
         &\int_{|t|< 1} f(t^{-1}x,0,0)d^\times t+\int_{1\le |t|\le |a|^{-1}} f(t^{-1}x,0,0)d^\times t\\
    &=\int_{|t|> 1} f(tx,0,0)d^\times t+\int_{1\ge |t|\ge|a|} f(tx,0,0)-f(0)d^\times t -\frac{\mathrm{vol}(K_{\GG_m})}{[F:\RR]}f(0)\log |a|.
    \end{align*}
    Observe that
    \begin{align*}
        \lim_{|a|\to 0} I(f)(ax)+\frac{\mathrm{vol}(K_{\GG_m})}{[F:\RR]}f(0)\log |a| 
    \end{align*}
    is well-defined, and we denote this as $\tilde{a}_2(f)(x)$. Explicitly,
    \begin{align}\label{eq:a2:arch}
    \begin{split}
        \tilde{a}_2(f)(x)&=\int_{|t|>1} f(0_{\underline{4}},0,t)d^\times t+ \int_{|t|\le 1} f(0_{\underline{4}},0,t)-f(0)d^\times t\\
        &+\int_{|t|> 1} f(tx,0,0)d^\times t+\int_{|t|\le 1} f(tx,0,0)-f(0)d^\times t.
    \end{split}
    \end{align}

    On the other hand,
    \begin{align}\label{eq:Z2compute}
        Z_2(f,s)=\int_{|t|\ge 1} f(0_{\underline{4}},0,t)|t|^s d^\times t+\int_{|t|<1} (f(0_{\underline{4}},0,t)-f(0))|t|^s d^\times t+\int_{|t|<1} f(0)|t|^s d^\times t.
    \end{align}
    The first two integrals are absolutely convergent for $\mathrm{Re}(s)>-1/2$. The last term above is $\frac{1}{[F:\RR]s}\mathrm{vol}(K_{\GG_m})f(0),$ so $\tilde{c}_2(I(f))=\frac{\mathrm{vol}(K_{\GG_m})}{[F:\RR]}f(0).$ This justifies (ii).
\end{proof}

\begin{lem}\label{lem:cequal}
    Suppose $F=\RR$ and $\ell$ is even, or $F=\CC$ and $\ell\ge 2$. Suppose $\psi(t)=e^{2\pi i\mathrm{tr}_{F/\RR}(t)}.$ Then
    \begin{align*}
        \tilde{c}_\ell(\tilde{b}_\ell)=c_\ell(b_\ell)\neq 0.
    \end{align*}
\end{lem}

\begin{proof}
     We have
    \begin{align*}
        Z_{\ell}(\tilde{b}_\ell,s)=\int_{F^\times} e^{-\pi [F:\RR]|t|_\RR^2}|t|^s d^\times t=\zeta(s).
    \end{align*}
    Assume $\ell$ is even and $F=\RR$. Then
    \begin{align*}
    \tilde{c}_\ell(\tilde{b}_\ell)=\mathrm{Res}_{s=2-\ell} \pi^{-s/2}\Gamma\left(\frac{s}{2}\right)=2\pi^{\frac{\ell-2}{2}}\frac{(-1)^{\frac{\ell-2}{2}}}{\left(\frac{\ell-2}{2}\right)!}.
    \end{align*}
    On the other hand by \cite[\S 7.2, (12), (37), (38)]{Highertrans:II}, the coefficient of $-\log |x|_\RR=-\log |x|$ in the asymptotic expansion of $b_\ell(x)=2\pi^{\frac{\ell-2}{2}}\widetilde{K}_{\frac{\ell-2}{2}}(2\pi|x|_\RR)$ is
    \begin{align*}
        c_\ell(b_\ell)=-2\pi^{\frac{\ell-2}{2}}\frac{(-1)^{\ell/2}}{\Gamma(\ell/2)}=\tilde{c}_\ell(\tilde{b}_\ell).
    \end{align*}

    Suppose $F=\CC$ and $\ell\ge 2$. Then
    \begin{align*}
\tilde{c}_\ell(\tilde{b}_\ell)=\mathrm{Res}_{s=2-\ell} 2(2\pi)^{-s}\Gamma(s)=2(2\pi)^{\ell-2}\frac{(-1)^{\ell-2}}{(\ell-2)!}.
    \end{align*}
In the asymptotic expansion of $b_\ell(x)=4(2\pi)^{\ell-2}\widetilde{K}_{\ell-2}(4\pi |x|_\RR)$, the coefficient of $-\log |x|=-2\log |x|_\RR$ is
    \begin{align*}
        \tilde{c}_\ell(b_\ell)=-2(2\pi)^{\ell-2}\frac{(-1)^{\ell-1}}{\Gamma(\ell-1)}=c_\ell(b_\ell).
    \end{align*}
    \end{proof}

\begin{prop}\label{prop:min:rep:arch}
Suppose $\ell \ge 2$. Then $I(\widetilde{\mathcal{S}}(X_\ell(F)))=\mathcal{S}(X_\ell(F)),$ and the map $I$ is $\mathrm{O}_{V_{\ell+1}}(F)$-equivariant. In particular, the diagram
$$
\begin{CD}
\widetilde{\mathcal{S}}(X_\ell(F)) @>{\widetilde{\mathcal{F}}_{X_\ell}}>> \widetilde{\mathcal{S}}(X_\ell(F))\\
@V{ I}VV @V{ I}VV
\\
\mathcal{S}(X_\ell(F)) @>{\mathcal{F}_{X_\ell}}>>  \mathcal{S}(X_\ell(F))
\end{CD}
$$
commutes. Moreover, for $\ell\ge 3$ we have a commutative diagram of $\mathrm{O}_{V_\ell}(F)$-modules 
    \begin{align*}
    \begin{CD}
    \widetilde{\mathcal{S}}(X_\ell(F)) @>{(\tilde{c}_\ell,\tilde{d}_{\ell})}>> \CC\oplus\widetilde{\mathcal{S}}(X_{\ell-1}(F))\\
    @V{I}VV @V{(\mathrm{id},I)}VV
    \\
    \mathcal{S}(X_\ell(F)) @>{(c_\ell,\varepsilon(\ell-2,\psi)d_{\ell})}>>  \CC\oplus\mathcal{S}(X_{\ell-1}(F)).
    \end{CD}
    \end{align*}
\end{prop}

\begin{proof}
   Suppose $\ell\ge 3$. Let $\mathcal{F}$ be the operator in Lemma \ref{lem:des}. We claim $\mathcal{F}(f)=\mathcal{F}_{X_\ell}(f)$ for $f\in \mathcal{S}_{\mathrm{ES}}(X_\ell(F))$. Assuming the claim, we have as vector spaces
    \begin{align*}
        I(\widetilde{\mathcal{S}}(X_\ell(F)))&=I(W_\ell)+I(\widetilde{\mathcal{F}}_{X_\ell}(W_\ell))=\mathcal{S}_{\mathrm{ES}}(X_\ell(F))+\mathcal{F}(\mathcal{S}_{\mathrm{ES}}(X_\ell(F)))\\
        &=\mathcal{S}_{\mathrm{ES}}(X_\ell(F))+\mathcal{F}_{X_\ell}(\mathcal{S}_{\mathrm{ES}}(X_\ell(F)))=\mathcal{S}(X_\ell(F)).
    \end{align*}
    The map $I: \widetilde{\mathcal{S}}(X_\ell(F))\longrightarrow \mathcal{S}(X_\ell(F))$ is continuous by the closed graph theorem, so $I(\widetilde{\mathcal{S}}(X_\ell(F)))=\mathcal{S}(X_\ell(F))$ as Fr\'echet spaces by the open mapping theorem. Now as each $f\in \mathcal{S}(X_\ell(F))$ is of the form $f=f_1+\mathcal{F}(f_2)$ for some $f_1,f_2\in \mathcal{S}_{\mathrm{ES}}(X_\ell(F)),$ we have
    \begin{align*}
        \mathcal{F}(f)&=\mathcal{F}(f_1+\mathcal{F}(f_2))=\mathcal{F}(f_1)+f_2=\mathcal{F}_{X_\ell}(f_1)+f_2\\
        &=\mathcal{F}_{X_\ell}(f_1+\mathcal{F}_{X_\ell}(f_2))=\mathcal{F}_{X_\ell}(f_1+\mathcal{F}(f_2))=\mathcal{F}_{X_\ell}(f).
    \end{align*}
    Therefore, $\mathcal{F}=\mathcal{F}_{X_\ell}$ and thus $I$ is $\mathrm{O}_{V_{\ell+1}}(F)$-equivariant.

    To prove the claim, we use a global argument. If $F=\RR$ (resp. $\CC$), choose $E=\QQ$ (resp. $\QQ[i]$). Let $f_\infty\in \mathcal{S}_{\mathrm{ES}}(X_\ell(F))$ and $\xi\in X_\ell^\circ(E).$ Since the local theory over nonarchimedean local fields agrees by results in \S \ref{sec:agree}, as argued in Lemma \ref{lem:des} we can choose finite places $v_1,v_2,$ and $f_{v_1}\in \mathcal{S}(X_\ell^\circ(E_{v_1})),$ $f_{v_2}\in \mathcal{F}_{X_\ell}(\mathcal{S}(X_\ell^\circ(E_{v_2})))$ and $f^{\infty v_1 v_2}\in \mathcal{S}(X_\ell(\A_E^{\infty v_1v_2}))$  so that additional terms in Theorem \ref{thm:Eis} and Theorem \ref{thm:Poisson:lift} vanish, $\mathcal{F}_{X_\ell}(f_{v_1}f_{v_2}f^{\infty v_1v_2})(\xi)\neq 0$, and
    \begin{align*}
        \mathcal{F}_{X_\ell}(f_\infty)\mathcal{F}_{X_\ell}(f_{v_1}f_{v_2}f^{\infty v_1v_2})(\xi)=\sum_{x\in X_\ell^\circ(E)}f_\infty f_{v_1}f_{v_2}f^{\infty v_1v_2}(x)=\mathcal{F}(f_\infty)\mathcal{F}_{X_\ell}(f_{v_1}f_{v_2}f^{\infty v_1v_2})(\xi).
    \end{align*}
    Therefore, $\mathcal{F}_{X_\ell}(f_\infty)(\xi)=\mathcal{F}(f_\infty)(\xi)$ for $\xi\in X_\ell^\circ(E).$ Since $X_\ell^\circ(E)$ is dense in $X_\ell^\circ(F),$ the claim follows by continuity. 

    For $\ell=2,$ by Theorem \ref{thm:asymplge3} and Corollary \ref{eq:dvanish} we have a continuous map
    \begin{align*}
        J:\mathcal{S}(X_{2}(F))\longrightarrow \mathcal{S}(X_{3}(F))/\mathcal{S}_{\mathrm{ES}}(X_{3}(F)) &\cong \widetilde{\mathcal{S}}(X_{3}(F))/W_3'\\
        &\xrightarrow{\,\,\tilde{d}_3\,\,} \widetilde{\mathcal{S}}(X_{2}(F))/\tilde{d}_3(W_3')\xrightarrow{\,\,I\,\,} I(\widetilde{\mathcal{S}}(X_{2}(F))). 
    \end{align*}
    We may assume $\psi(t)=e^{2\pi i\mathrm{tr}_{F/\RR}(t)}.$ Then $J(b_2)=b_2$ and $J\circ \sigma_2=\sigma_2'\circ J$. Since $\mathcal{S}(X_{2}(F))$ is irreducible, by the Iwasawa decomposition $J$ is the inclusion map. Therefore, $\mathcal{F}$ and $\mathcal{F}_{X_{2}}$ agree on $\mathcal{S}(X_{2}(F))$. In particular, $\mathcal{F}$ is unitary and thus $\mathcal{F}=\mathcal{F}_{X_{2}}$ on $L^2(X_{2}(F)).$ Since $I(\widetilde{\mathcal{S}}(X_{2}(F)))$ is contained in the set of smooth vectors $\mathcal{S}(X_{2}(F))$ in $L^2(X_{2}(F)),$ we have $J$ is the identity map. This justifies the first statement.

    The second assertion can be proved similarly as in  Proposition \ref{prop:min:rep} using Theorem \ref{thm:asymplge3}, Lemma \ref{lem:easy:arch} and Lemma \ref{lem:cequal}. We leave it to the reader.
\end{proof}

We equip $\mathcal{S}(X_{1}(F)):=I(\widetilde{\mathcal{S}}(X_1(F)))$ with the quotient topology so that it is a Fr\'echet space.

\begin{prop}\label{prop:case1:arch} 
  We have natural $\mathrm{O}_{V_1}(F)$-equivariant maps
    \begin{align*}
        \mathcal{S}(X_{1}(F))\xrightarrow{( c_1, (d_1',d_1))} \CC\oplus \mathrm{Ind}_{\mathrm{SO}_{V_1}}^{\mathrm{O}_{V_1}}\CC_1,
    \end{align*}
    where $(c_1,d_1',d_1)$ are normalized so that
    \begin{align*}
        (d_1',d_1)(b_1)=(1,1), \quad c_1(b_1)=\frac{b_1(0)}{\zeta(1)}=1.
    \end{align*}
    Furthermore, we have commutative diagrams of $\mathrm{SO}_{V_1}(F)$-modules
     \begin{align*}
    \begin{minipage}{0.48\textwidth}
    \vspace{0pt}
    \[
    \begin{CD}
    \widetilde{\mathcal{S}}(X_1(F)) @>{\quad\quad(\tilde{c}_1,\tilde{d}_{1})\quad\quad}>> \CC\oplus \widetilde{\mathcal{S}}(X_0(F))\\
    @V{I}VV @V{(\mathrm{id},I)}VV
    \\
    \mathcal{S}(X_{1}(F)) @>{(c_1,\varepsilon(-1,\psi)\zeta(2)d_{1})}>>  \CC\oplus \CC_{1}
    \end{CD}
\]
\end{minipage}
\hfill
\begin{minipage}{0.48\textwidth}
\vspace{6pt}
\[
    \begin{CD}
    \widetilde{\mathcal{S}}(X_1(F)) @>{\quad\quad \tilde{d}_{1}\quad\quad}>> \widetilde{\mathcal{S}}(X_0(F))\\
    @V{I}VV @V{\mathrm{ev}}VV
    \\
    \mathcal{S}(X_{1}(F)) @>{\quad\quad \varepsilon(-1/2,\psi)d_{1}'\quad\quad}>>  \CC_{1}.
    \end{CD}
    \]
\end{minipage}
    \end{align*}
    
    \medskip
    \noindent Here $\mathrm{ev}$ is the evaluation map of functions on $F^2$ at $(0,0).$
\end{prop}

\begin{proof}
Let $f\in \mathcal{S}(V_2(F))$ that is $r_1(K)$-invariant. By \eqref{eq:Tate}, we have
\begin{align*}
    I(\tilde{d}_1(f))&=\int_{F^\times}\int_F|t|^{2}f(0,y,0,0)\psi(ty)dy d^\times t\\
    &=\zeta(2)\varepsilon(-1,\psi)\lim_{s\to -1}\frac{1}{\zeta(s)}\int_{F^\times} |t|^{s}f(0,t,0,0) d^\times t.
\end{align*}
Since $f$ is $r_1(K)$-invariant, the expression above is invariant under changing variables $t\mapsto tc$ for $c\in K_{\GG_m}.$ Thus we may assume $f(0,\cdot,0,\cdot)$ is invariant under $K_{\GG_m}$ in both entries.

Assume $F=\RR$. Then
\begin{align*}
    &\int_{F^\times} |t|^{s}f(0,t,0,0) d^\times t\Bigg|_{s=-1}\\
    &=\int_{|t|\ge 1} |t|^{-1}f(0,t,0,0)d^\times t-\mathrm{vol}(K_{\GG_m})f(0)+\int_{|t|<1} |t|^{-1}(f(0,t,0,0)-f(0))d^\times t\\
    &= \int_{F^\times} |t|^{-1}f(0,t,0,0)-|t|^{-1}f(0) d^\times t.
\end{align*}
Here we have used the fact that $f(0,t,0,0)-f(0)=O(|t|^2)$ for $|t|<r$ small. Write this as
\begin{align*}
    \int_{F^\times} \lim_{|a|\to 0} |t|^{-1}\left(f(0,t,0,at^{-1})-f(0,0,0,at^{-1})\right) d^\times t.
\end{align*} Then we can apply Lebesgue's dominated convergence theorem to see this is
\begin{align*}
    &\lim_{|a|\to 0}\int_{F^\times} |t|^{-1}\left(f(0,t,0,at^{-1})-f(0,0,0,at^{-1})\right) d^\times t=\lim_{|a|\to 0}\frac{1}{|a|} \left(I(f)(0,a)-I(f)(0)\right),
\end{align*}
where the last equality follows from changing variables $t\mapsto ta.$ We define
\begin{align*}
    d_1(I(f)):=\frac{1}{\zeta(-1)}\lim_{|a|\to 0}\frac{1}{|a|} \left(\frac{I(f)(0,a)+I(f)(0,-a)}{2}-I(f)(0)\right).
\end{align*}

Now suppose $F=\CC.$ Let 
\begin{align*}
    b_{1}:=\frac{\partial^2}{\partial t\partial \bar{t}}f(0,t,0,0)\bigg|_{t=0}.
\end{align*}
Then using $f(0,t,0,0)=f(0)+b_1|t|+O(|t|^2)$ for $|t|$ small, we have
\begin{align*}
    &\lim_{s\to -1}\frac{1}{\zeta(s)}\int_{F^\times} |t|^{s}f(0,t,0,0) d^\times t\\
    &=\lim_{s\to -1} \frac{1}{\zeta(s)}\int_{|t|<1} |t|^{s}f(0,t,0,0) d^\times t\\
    &=\lim_{s\to -1} \frac{1}{\zeta(s)}\int_{|t|<1} |t|^{s}\left(f(0,t,0,0)-f(0)-b_1|t|\right)+f(0)|t|^s+b_1|t|^{s+1} d^\times t\\
    &=\frac{1}{2\mathrm{Res}_{s=-1}\zeta(s)} b_1.
\end{align*}
On the other hand, using again the fact that $f(0,a,0,t^{-1})-f(0,0,0,t^{-1})=O(|a|)g(t^{-1})$ for some $g\in \mathcal{S}(F)$ for $0<|a|<1$ sufficiently small, we have
\begin{align*}
    &\frac{1}{-|a|\log |a|}\left(I(f)(0,a)-I(f)(0\right)) \\&=\frac{1}{-|a|\log |a|}\int_{F^\times}|t|^{-1}\left(f(0,at,0,t^{-1})-f(0,0,0,t^{-1})\right) d^\times t\\
    &=\frac{1}{-\log |a|}\int_{F^\times}|t|^{-1}\left(f(0,t,0,at^{-1})-f(0,0,0,at^{-1})\right) d^\times t\\
    &=\frac{1}{-\log |a|}\int_{|a|<|t|<1}|t|^{-1}\left(f(0,t,0,at^{-1})-f(0,0,0,at^{-1})\right) d^\times t+O\left(\frac{1}{-\log |a|}\right).
\end{align*}
The limit as $|a|\to 0$ exists, which equals
\begin{align*}
    &\lim_{|a|\to 0}\frac{1}{-\log |a|}\int_{|a|<|t|<1}\frac{\partial^2}{\partial v_2\partial \bar{v}_2}f(0,0,0,at^{-1})d^\times t\\
    &=\lim_{|a|\to 0}\frac{1}{-\log |a|}\int_{|a|<|t|<1}\frac{\partial^2}{\partial v_2\partial \bar{v}_2}f(0,0,0,t)d^\times t=b_1/2.
\end{align*}
We define
\begin{align*}
    d_1(I(f)):=\frac{1}{\mathrm{Res}_{s=-1} \zeta(s)}\lim_{|a|\to 0}\frac{1}{-|a|\log |a|} \left(\frac{\int_{K_{\GG_m}}I(f)(0,ac)dc}{\mathrm{vol}(K_{\GG_m})}-I(f)(0)\right).
\end{align*}
Then in both cases, by the definition of $d_1$ we have $I\circ \tilde{d}_1=\varepsilon(-1,\psi)\zeta(2)d_1\circ I.$

The map $d_1'$ can be defined analogously by considering the asymptotic expansion of $I(f)(a,0)$. The same proof as in Proposition \ref{prop:case1} implies $\mathrm{ev}\circ\tilde{d}_1=\varepsilon(-1/2,\psi)d_1'\circ I.$ Define
\begin{align*}
    c_1(I(f))&:=\frac{1}{\zeta(1)}I(f)(0).
\end{align*}
Then $c_1(I(f))=\tilde{c}_1(f)$ by Lemma \ref{lem:easy:arch}(iii). This completes the proof.

\end{proof}

We continue to use notations in the discussion preceding Proposition \ref{prop:case2} for the following proposition.

\begin{prop}\label{prop:case2:arch}
\begin{enumerate}
\item[(i)]
We have a natural $\mathrm{O}_{V_2}(F)$-equivariant map
    \begin{align*}
       \mathcal{S}(X_{2}(F))/\mathcal{S}_{\mathrm{ES}}(X_{2}(F))\longrightarrow \pi \CC[[\mathrm{Res}_{F/\RR}X_{2}]]
    \end{align*}
    where $\pi$ is of length $2$ and admits a nonsplit exact sequence of $\mathrm{O}_{V_2}(F)$-modules
    \begin{align*}
        0\longrightarrow \sigma \longrightarrow \pi \longrightarrow \CC\longrightarrow 0.
    \end{align*}
    Here the map $\pi\to \CC$ is given by taking the highest order term in the germ expansion.
    Let 
    \begin{align*}
      c_2:  \mathcal{S}(X_{2}(F))\xrightarrow{\,\,d_2\,\,}\pi\lto \CC.
    \end{align*}
    Then $c_2(I(f))=\tilde{c}_2(f)$.
    
   \item[(ii)] We have an isomorphism $J:\pi\cong \mathcal{S}(X_{1}(F))$ of $P_2(F)$-representations such that $J(d_2(b_2))=b_1$. Under this identification we have a commutative diagram of  $P_2(F)$-representations: $$
    \begin{CD}
    \widetilde{\mathcal{S}}(X_2(F)) @>{\tilde{d}_{2}}>> \widetilde{\mathcal{S}}(X_{1}(F))\\
    @V{I}VV @V{I}VV
    \\
    \mathcal{S}(X_{2}(F)) @>{\varepsilon(0,\psi)d_{2}}>>  \mathcal{S}(X_1(F)).
    \end{CD}
    $$
    \end{enumerate}
\end{prop}

\begin{proof}
     Arguing as in Proposition \ref{prop:case2}, by Lemma \ref{lem:easy:arch}(ii) and Proposition \ref{prop:min:rep:arch}, by taking germs at the origin of functions in $\mathcal{S}(X_{2}(F))$ we have a natural  map 
    \begin{align*}
          \mathcal{S}(X_{2}(F))/\mathcal{S}(X_{2}^\circ(F))&\longrightarrow \pi''\\
          I(f)&\mapsto \tilde{a}_2(f)(\cdot)-\tilde{c}_2(f)\log|\cdot|.
    \end{align*}
  The representation $\pi''$ fits into the following nonsplit exact sequence of $\mathrm{O}_{V_2}(F)$-representations
    \begin{align*}
        &0 \longrightarrow \pi' \longrightarrow \pi''\longrightarrow \CC\longrightarrow 0,\\
       & 0 \longrightarrow \CC \longrightarrow \pi'\longrightarrow \mathrm{Ind}_{\mathrm{SO}_{V_2}}^{\mathrm{O}_{V_2}}(\mathrm{St} \otimes \CC)\longrightarrow 0.
    \end{align*}
    Take $\pi:=\pi''/\CC\otimes\CC.$ By the differential action of $\mathrm{Lie}\,\mathrm{O}_{V_3}(F)$ on the minimal representation  \cite{Kobayashi:Mano, minrep:real}, we have $\mathcal{S}(X_2(F))$ is stable under multiplication by the coordinate functions $\CC[\mathrm{Res}_{F/\RR}X_{2}],$ so we can extend the results by tensoring $\CC[[\mathrm{Res}_{F/\RR}X_{2}]]$ at the germs. 
    
    By the asymptotics in Lemma \ref{lem:easy:arch}(ii) and well-known results on principal series of $\SL_2(F)$, the space of continuous $\mathrm{O}_{V_2}(F)$-invariant linear functionals $\mathcal{S}(X_2(F))/\mathcal{S}(X_2^\circ(F))\rightarrow \CC$ is one-dimensional (see also the remark below). Therefore, $\tilde{c}_2(f)=c_2(I(f))$
    by Lemma \ref{lem:cequal}. This proves (i).
    
    For (ii) by Lemma \ref{lem:ddescent} for $f\in \widetilde{\mathcal{S}}(X_{2}(F))$ and $x\in X_{1}^\circ(F),$
         \begin{align}\label{eq:compute}
        I(\tilde{d}_2(f))(x)=\lim_{|a|\to 0} \int_{F} I(f)(ax,0,ay) \psi(y)dy.
    \end{align}
    This vanishes if $f\in W_{2}'$ by Corollary \ref{eq:dvanish}. Furthermore, if $I(f)(v)=O(|v|^{\epsilon})$ as $|v|\to 0$ for some $\epsilon>0,$ then since $I(f)(x,0,y)$ as a function in $y$ behaves like a Schwartz function in $\mathcal{S}(F)$ as $|y|\to \infty$, we have for $\epsilon/(\epsilon+1)>\delta>0$
    \begin{align*}
        &\int_{F} I(f)(ax,0,ay) \psi(y)dy\\
        &=|a|^{-1}\left(\int_{|y|\le |a|^{1-\delta}} I(f)(ax,0,y)\psi(a^{-1}y)dy+\int_{|a|^{1-\delta}<|y|} I(f)(ax,0,y)\psi(a^{-1}y)dy\right)\\
        &=O(|a|^{-\delta+\epsilon(1-\delta)})
    \end{align*}
    converges to $0$ as $|a|\to 0.$ Thus $I(\tilde{d}_2(f))=0,$ and the map $I\circ \tilde{d}_2$ factors through 
    \begin{align*}
        \widetilde{\mathcal{S}}(X_2(F))\xrightarrow{\quad I\quad} \mathcal{S}(X_2(F))\xrightarrow{\quad  d_2\quad } \pi.
    \end{align*}
    Hence we have an induced $P_2(F)$-equivariant surjection $\tilde{J}:\pi\to \mathcal{S}(X_1(F)).$ 

    Elements in $\sigma$ can be viewed as functions on $(\mathbb{P}^1(F)\times \{\mathrm{pt}\})\cup (\{\mathrm{pt}\}\times\mathbb{P}(F))$. In particular, $\sigma$ contains a $P_2(F)$-stable subspace $\sigma'$ homeomorphic to $\mathcal{S}((N(F)-\{0\})^2)=\mathcal{S}((F^\times)^2)$. Arguing similarly as in Lemma \ref{lem:Iact}, by Fourier inversion and the fact that $N_2(F)$ acts trivially on $\sigma/\sigma',$ we have $\sigma'$ is the unique (topologically) irreducible $P_2(F)$-subrepresentation of $\sigma$. Furthermore, by \eqref{eq:compute} $\tilde{J}(\sigma')=\mathcal{S}((F^\times)^2).$ Thus $\tilde{J}$ is injective and hence an isomorphism. When $\psi(t)=e^{2\pi i\mathrm{tr}_{F/\RR}(t)},$ $\tilde{J}(d_2(b_2))=b_1.$ For general $\psi,$ the argument in Proposition \ref{prop:case2} applies.
\end{proof}

\begin{rem}
One can give an alternative definition of $\mathcal{S}(X_2(F))$ as in \cite[\S 5.2]{Getz:Hsu:Leslie} and mimic the argument in \cite[\S 5]{Hsu:asymptotics} to show that the map in Proposition \ref{prop:case2:arch}(i) can be upgraded to the following exact sequence of smooth $\mathrm{O}_{V_2}(F)$-modules
        \begin{align*}
        0\longrightarrow\mathcal{S}_{\mathrm{ES}}(X_{2}(F))\longrightarrow\mathcal{S}(X_{2}(F))\longrightarrow \pi \CC[[\mathrm{Res}_{F/\RR}X_{2}]]\longrightarrow 0.
    \end{align*}
     % Thus, we only give a sketch for the stronger statement. For a unitary character $\chi:F^\times\to \CC^\times,$ let $L(s,\chi)$ be the Tate's $L$-function. Using the well-known result on the principal series of $\SL_2(F)$ and the explicit description of the basic function $b_2$,  Explicitly, one can show $f\in \mathcal{S}(X_2(F))$ if and only if both
    % \begin{align*}
    %     \frac{f_{\chi_s}}{L(s+1,\chi)^2} \quad \textrm{ and }\quad \frac{M_{w_0}(\chi_s)\otimes M_{w_0}(\chi_s)(f_{\chi_s})}{L(s,\chi)^2}
    % \end{align*}
    % are holomorphic sections. Then the exact sequence in the statement can be obtained (c.f. 
\end{rem}

\begin{cor}\label{cor:2eq1:arch}
For $f\in \mathcal{S}(X_{2}(F))$, $c_2(f)=2c_1(d_2(f))$.
\end{cor}

\begin{proof}
    We may assume $\psi(t)=e^{2\pi i\mathrm{tr}_{F/\RR}(t)}.$ By the computation in Lemma \ref{lem:cequal},
    \begin{align*}
        c_2(b_2)= 2=2c_{1}(b_1)=2c_1(d_2(b_2)).
    \end{align*}
    Since both $c_2$ and $c_1\circ d_2$ are continuous $\mathrm{O}_{V_2}(F)$-equivariant linear functionals on $\mathcal{S}(X_2(F))/\mathcal{S}(X_2^\circ(F))$, we have $c_2=2c_1\circ d_2.$
\end{proof}

 For $f\in \mathcal{S}(X_2(F)),$ choose $\tilde{f}\in \widetilde{\mathcal{S}}(X_2(F))$ such that $I(\tilde{f})=f$. As in the nonarchimedean case, we have a $P_2(F)$-equivariant morphism
\begin{align*} 
  \mathcal{S}(X_{2}(F))\xrightarrow{(a_2, d_2)} (\CC\otimes\CC)\oplus \pi,
\end{align*}
where
\begin{align*}
    a_2(f):=\tilde{a}_2(\tilde{f})(0:0:0:1).
\end{align*}
In particular,  $a_2(f)=f(0)$ for $f\in \mathcal{S}_{\mathrm{ES}}(X_{2}(F)).$

\begin{lem}\label{lem:der:arch}
Let $f\in \mathcal{S}(V_2(F)\oplus F^2)$. We have 
\begin{align*}
    &\frac{d}{ds} \left(-\frac{Z_{2}(f,s)}{\zeta(s)}+\frac{Z_{1}(\tilde{d}_2(f), s+1)}{\varepsilon(-s,\psi)\zeta(s+1)}\right)\Bigg|_{s=0}\\
    &=-\frac{d^2}{d^2s}\zeta(s)^{-1}\bigg|_{s=0}c_2(I(f))-\frac{1}{2}a_2(I(f)).
\end{align*}
\end{lem}

\begin{proof}
Note that for $F$ real or complex,
\begin{align*}
    \frac{d}{ds}\zeta(s)^{-1}\bigg|_{s=0}=\lim_{s\to 0} \frac{1}{s\zeta(s)}=\frac{1}{2}.
\end{align*}
We can assume $f$ is $r_2(K)$-invariant. By \eqref{eq:Z2compute}  we have
    \begin{align*}
        \frac{d}{ds}\frac{Z_2(f,s)}{\zeta(s)}\bigg|_{s=0}&=\frac{1}{2}\frac{d^2}{d^2s}\zeta(s)^{-1}\bigg|_{s=0}f(0)\frac{\mathrm{vol}(K_{\GG_m})}{[F:\RR]}\\
        &+\frac{1}{2}\left(\int_{|t|\ge1} f(0_{\underline{4}},0,t) d^\times t+\int_{|t|<1} (f(0_{\underline{4}},0,t)-f(0)) d^\times t\right).
    \end{align*}
    On the other hand, by \eqref{eq:Tate}
    \begin{align*}
 Z_1(\tilde{d}_2(f),s+1)&=\varepsilon(-s,\psi)\frac{\zeta(1+s)}{\zeta(-s)}\int_{F^\times} f(0_{\underline{2}},0,t,0,0)|t|^{-s}d^\times t.
    \end{align*}
    Thus similarly
    \begin{align*}
        &\frac{d}{ds} \frac{Z_1(\tilde{d}_2(f),s+1)}{\varepsilon(-s,\psi)\zeta(s+1)}\Bigg|_{s=0}\\
        &=-\frac{1}{2}\frac{d^2}{ds^2}{\zeta(s)}^{-1}\bigg|_{s=0}f(0)\frac{\mathrm{vol}(K_{\GG_m})}{[F:\RR]}\\
        &-\frac{1}{2}\left(\int_{|t|>1} f(0_{\underline{2}},0,t,0,0)d^\times t+\int_{|t|\le 1} (f(0_{\underline{2}},0,t,0,0)-f(0))d^\times t\right).
    \end{align*}
By the proof of Lemma \ref{lem:easy:arch} and \eqref{eq:a2:arch}, $\tilde{c}_2(f)=f(0)\frac{\mathrm{vol}(K_{\GG_m})}{[F:\RR]}$ and
\begin{align*}
    \tilde{a}_2(f)(0:0:0:1)=&\int_{|t|>1} f(0_{\underline{4}},0,t)d^\times t+ \int_{|t|\le 1} f(0_{\underline{4}},0,t)-f(0)d^\times t\\
    &+  \int_{|t|< 1} f(0_{\underline{2}},0,t^{-1},0,0)d^\times t+\int_{1\ge |t|} (f(0_{\underline{2}},0,t,0,0)-f(0))d^\times t.
\end{align*}
Therefore, by the definition of $c_2$ and $a_2$ we have
\begin{align*}
    &\frac{d}{ds} \left(-\frac{Z_{2}(f,s)}{\zeta(s)}+\frac{Z_{1}(\tilde{d}_2(f), s+1)}{\varepsilon(-s,\psi)\zeta(s+1)}\right)\Bigg|_{s=0}=-\frac{d^2}{d^2s}\zeta(s)^{-1}\bigg|_{s=0}c_2(I(f))-\frac{1}{2}a_2(I(f)).
\end{align*}
\end{proof}

\section{Proof of Theorem \ref{thm:main}}\label{sec:mainproof}

Let $E$ be a number field. For $\ell\ge 1,$ define
\begin{align*}
    d_\ell&:=\otimes_{v} d_{\ell,v}:\mathcal{S}(X_\ell(\A_E))\longrightarrow \mathcal{S}(X_{\ell-1}(\A_E)),\\
    d_1'&:=\otimes_{v} d_{1,v}':\mathcal{S}(X_1(\A_E))\longrightarrow \mathcal{S}(X_{0}(\A_E)).
\end{align*}
Here $\mathcal{S}(X_0(\A_E)):=\CC_{1}$ is the representation $|\cdot|$ of $\A_E^\times.$ Note that $d_1'=d_1\circ w$
where $w:\mathcal{S}(X_1(\A_E))\longrightarrow\mathcal{S}(X_1(\A_E))$ is the automorphism induced by the action of $\begin{psmatrix}
    0 & 1\\
     1 & 0
\end{psmatrix}$ that permutes the two entries of $X_1$.

For $\ell\ge 3,$ when $v$ is finite, we have $c_{\ell,v}(b_{\ell,v})=\zeta_v(2-\ell).$ Thus we normalize $c_{\ell,v}$ by
\begin{align*}
    c_{\ell,v}^\circ:=\begin{cases}
        \zeta_v(2-\ell)^{-1}c_{\ell,v} & \textrm{if } v\nmid\infty \textrm{ or } v \textrm{ is real and } 2\nmid\ell,  \\
        \frac{\zeta_v(s)^{-1}}{s+\ell-2}\bigg|_{s=2-\ell} c_{\ell,v} & \textrm{otherwise,}
    \end{cases}
\end{align*}
and we define
\begin{align*}
     c_\ell&:=\otimes_{v} c_{\ell,v}^\circ:\mathcal{S}(X_\ell(\A_E))\longrightarrow \CC.
\end{align*}
For $\ell=1,2,$ $c_\ell(b_{\ell,v})=1$ for finite $v$, so we let
\begin{align*}
     c_\ell&:=\otimes_{v} c_{\ell,v}:\mathcal{S}(X_\ell(\A_E))\longrightarrow \CC.
\end{align*}

Recall that $\epsilon(s)=|D|^{\frac{1}{2}-s},$ where $D\in \ZZ$ is the absolute discriminant of $E$. We summarize the results in \S \ref{sec:agree} and \S \ref{sec:agree:arch} in the adelic setting.
\begin{thm}\label{thm:summary}
For $\ell\ge 2,$ $I\circ \widetilde{\mathcal{F}}_{X_\ell}=\mathcal{F}_{X_\ell}\circ I$ and
\begin{align*}
    I\circ \tilde{d}_{\ell}=|D|^{\frac{5}{2}-\ell}d_{\ell}\circ I.
\end{align*}
For $\ell\ge 3$, $$\tilde{c}_{\ell}=\zeta(2-\ell)c_\ell\circ I=|D|^{\ell-\frac{3}{2}}\zeta(\ell-1)c_\ell\circ I.$$ Moreover,
\begin{align*}
  I\circ \tilde{d}_{1}=\zeta(2)|D|^\frac{3}{2}d_{1}\circ I, \quad\quad \mathrm{ev}\circ \tilde{d}_1= |D|d_1'\circ I, \quad\quad  c_2=2^{|\infty|}c_1\circ d_2.
\end{align*}
Here $\mathrm{ev}$ is the evaluation map of functions on $\A_E^2$ at $(0,0).$
\end{thm}

\begin{proof}
    Only the identity $\tilde{c}_{\ell}=\zeta(2-\ell)c_\ell\circ I$ for $\ell\ge 3$ is unclear. Suppose $f=\otimes f_v\in \mathcal{S}(V_\ell(\A_E)\oplus \A_E^2).$ Since
    \begin{align*}
        \frac{Z_\ell(f,s)}{\zeta(s)}=\prod_{v} \frac{Z_\ell(f_v,s)}{\zeta_v(s)}
    \end{align*}
    whose factors are $1$ at almost all $v$, we have
    \begin{align*}
        \frac{\tilde{c}_\ell(f)}{\zeta(2-\ell)}= \left(\prod_{v\in S} \tilde{c}_{\ell,v}(f_v)\frac{\zeta_v(s)^{-1}}{s+\ell-2}\right)\bigg|_{s=2-\ell}\left(\prod_{v\not\in S} \frac{\tilde{c}_{\ell,v}(f_v)}{\zeta_v(2-\ell)}\right)
    \end{align*}
    where $S$ is the set of infinite places such that $\zeta(s)$ has a (simple) pole at $s=2-\ell.$ By the local computations in \S \ref{sec:agree} and \S \ref{sec:agree:arch}, this is
    \begin{align*}
        \left(\prod_{v\in S} c_{\ell,v}(I(f_v))\frac{\zeta_v(s)^{-1}}{s+\ell-2}\right)\bigg|_{s=2-\ell}\left(\prod_{v\not\in S} \frac{c_{\ell,v}(I(f_v))}{\zeta_v(2-\ell)}\right)=c_\ell(f).
    \end{align*}
    The identity for general $f$ follows by continuity of $c_\ell$ and linearity.
\end{proof}

For $\ell>i\ge 0,$ define
\begin{align*}
    d_{\ell, i}:=d_{i+1}\circ\cdots \circ d_\ell:\mathcal{S}(X_\ell(\A_E))\longrightarrow\mathcal{S}(X_i(\A_E)).
\end{align*}
Let $d_{\ell,\ell}$ be the identity map. Then $I\circ \tilde{d}_{\ell,i}=|D|^{\frac{(4-\ell-i)(\ell-i)}{2}}d_{\ell,i}\circ I$ for $i\ge 1,$ and $I\circ \tilde{d}_{\ell,0}=\zeta(2)|D|^{\frac{(4-\ell)\ell}{2}}d_{\ell,0}\circ I$. 

For $f=\otimes f_v\in \mathcal{S}(X_{2}(\A_E)),$ define
\begin{align}\label{eq:adelica2}
\begin{split}
    a_{2}(f):=&\sum_{v|\infty} \left(-\frac{d^2}{d^2s}\zeta_v(s)^{-1}\bigg|_{s=0}c_{2,v}(f_v)-\frac{1}{2}a_{2,v}(f_v)\right)\prod_{v'\neq v} c_{2,v'}(f_{v'})\\
    &+\frac{1}{2}\sum_{v\nmid \infty} (\log q_v)\big(c_{2,v}(f_v)-a_{2,v}(f_v)\big)\prod_{v'\neq v} c_{2,v'}(f_{v'}).
\end{split}
\end{align}
Note that the sum is actually finite since for $v\nmid\infty,$ $c_2(b_{\ell,v})=a_2(b_{\ell,v})=1$. Extend the definition by continuity and linearity to a linear functional
\begin{align*}
    a_2:\mathcal{S}(X_2(\A_E))\longrightarrow \CC.
\end{align*}

\begin{proof}[Proof of Theorem \ref{thm:main}]

By Theorem \ref{thm:Poisson:lift} and Theorem \ref{thm:summary} it remains to show for $f=\otimes f_v\in \mathcal{S}(V_3(\A_E))$
\begin{align*}
    &\tilde{c}_2(f)+\tilde{c}_1(\tilde{d}_{2}(f))=2^{1-|\infty|}\kappa' c_2(I(f))+2^{1-|\infty|}|D|^{\frac{1}{2}}\kappa a_2(I(f))
\end{align*}
where
\begin{align*}
    \kappa'=\frac{d}{ds}s\zeta(s)\bigg|_{s=0}.
\end{align*}
We begin by unwinding the definition of $\tilde{c}_1$ and $\tilde{c}_2$. Recall $\kappa=\mathrm{Res}_{s=1} \zeta(s).$ Using the functional equation $\zeta(s)=\epsilon(s)\zeta(1-s),$ we have by Proposition \ref{prop:case2} and Proposition \ref{prop:case2:arch}
\begin{align*}
    \tilde{c}_2(f)&= \left(\frac{d}{ds} s\zeta(s)\frac{Z_{2}(f,s)}{\zeta(s)}\right)\Bigg|_{s=0}\\
&=\big(\mathrm{Res}_{s=0}\zeta(s)\big)\left(\frac{d}{ds}\frac{Z_2(f,s)}{\zeta(s)}\right)\Bigg|_{s=0}+\kappa'\frac{Z_2(f,s)}{\zeta(s)}\Bigg|_{s=0}\\
    &=-|D|^{1/2}\kappa\left(\frac{d}{ds}\frac{Z_2(f,s)}{\zeta(s)}\right)\Bigg|_{s=0}+2^{-|\infty|}\kappa' c_2(I(f)).
\end{align*}
Similarly, by Proposition \ref{prop:case1}, Proposition \ref{prop:case2}, Proposition \ref{prop:case1:arch} and Proposition \ref{prop:case2:arch}
\begin{align*}
   \tilde{c}_1(\tilde{d}_{2}(f))&=  \left(\frac{d}{ds} s\varepsilon(-s)\zeta(1+s)\frac{Z_1(\tilde{d}_2(f),s+1)}{\varepsilon(-s)\zeta(1+s)}\right)\Bigg|_{s=0}\\
    &=|D|^{1/2}\kappa\left(\frac{d}{ds}\frac{Z_1(\tilde{d}_2(f),s+1)}{\varepsilon(-s)\zeta(1+s)}\right)\Bigg|_{s=0}+\left(\frac{d}{ds}s\zeta(-s)\right)\bigg|_{s=0} \left(\frac{Z_1(\tilde{d}_2(f),s+1)}{\varepsilon(-s)\zeta(1+s)}\right)\Bigg|_{s=0}\\
    &=|D|^{1/2}\kappa\left(\frac{d}{ds}\frac{Z_1(\tilde{d}_2(f),s+1)}{\varepsilon(-s)\zeta(1+s)}\right)\Bigg|_{s=0}+\kappa'  c_1(d_2(I(f))).
\end{align*}
We claim
\begin{align*}
    a_2(I(f))&=\frac{d}{ds}\left(-\frac{Z_2(f,s)}{\zeta(s)}+\frac{Z_1(\tilde{d}_2(f),s+1)}{\varepsilon(-s)\zeta(1+s)}\right)\Bigg|_{s=0}.
\end{align*}
Since $c_2=2^{|\infty|}c_1\circ d_2$, the assertion follows from the claim.

Let $S$ be a finite set of places including $\infty$ such that for $v\not\in S,$ $f_v=\one_{V_3(\mathcal{O}_v)}$ and $\psi_v$ is unramified. Then
\begin{align*}
    -\frac{Z_2(f,s)}{\zeta(s)}+\frac{Z_1(\tilde{d}_2(f),s+1)}{\varepsilon(-s)\zeta(1+s)}=-\prod_{v\in S}\frac{Z_2(f_v,s)}{\zeta_v(s)}+\prod_{v\in S}\frac{Z_1(\tilde{d}_2(f_v),s+1)}{\varepsilon(-s,\psi_v)\zeta_v(1+s)}.
\end{align*}
By the product rule and the identity
\begin{align*}
   \frac{Z_2(f_v,s)}{\zeta_v(s)}\bigg|_{s=0}= 2^{-\delta_{v|\infty}}c_{2,v}(I(f_v))= c_{1,v}\circ d_{2,v}(I(f_v))=\frac{Z_1(\tilde{d}_2(f_v),s+1)}{\varepsilon(-s,\psi_v)\zeta_v(1+s)}\bigg|_{s=0}
\end{align*}
used above, the claim follows from Lemma \ref{lem:der} and Lemma \ref{lem:der:arch}. This completes the proof.
\end{proof}

% ----------------------------------------------------------------

\bibliography{refs}{}

@article {AG:Nash,
    AUTHOR = {Aizenbud, Avraham and Gourevitch, Dmitry},
     TITLE = {Schwartz functions on {N}ash manifolds},
   JOURNAL = {Int. Math. Res. Not. IMRN},
  FJOURNAL = {International Mathematics Research Notices. IMRN},
      YEAR = {2008},
    NUMBER = {5},
     PAGES = {Art. ID rnm 155, 37},
      ISSN = {1073-7928},
   MRCLASS = {46T30 (14P20 46F05)},
  MRNUMBER = {2418286},
MRREVIEWER = {Michael Kunzinger},
       DOI = {10.1093/imrn/rnm155},
       URL = {https://doi.org/10.1093/imrn/rnm155},
}

@article{Getz:Hsu,
     author = {Getz, Jayce R. and Hsu, Chun-Hsien},
     title = {The {Fourier} transform for triples of quadratic spaces},
     journal = {Annales de l'Institut Fourier},
     year = {2026},
     publisher = {Association des Annales de l{\textquoteright}institut Fourier},
     doi = {10.5802/aif.3765},
     language = {en},
     note = {Online first},
}

@article {BKnormalized,
    AUTHOR = {Braverman, Alexander and Kazhdan, David},
     TITLE = {Normalized intertwining operators and nilpotent elements in
              the {L}anglands dual group},
      NOTE = {Dedicated to Yuri I. Manin on the occasion of his 65th
              birthday},
   JOURNAL = {Mosc. Math. J.},
  FJOURNAL = {Moscow Mathematical Journal},
    VOLUME = {2},
      YEAR = {2002},
    NUMBER = {3},
     PAGES = {533--553},
      ISSN = {1609-3321},
   MRCLASS = {22E50 (11F70 22E55)},
  MRNUMBER = {1988971},
MRREVIEWER = {Goran Mui\"A},
}

@article {Getz:Liu:BK,
    AUTHOR = {Getz, Jayce Robert and Liu, Baiying},
     TITLE = {A refined {P}oisson summation formula for certain
              {B}raverman-{K}azhdan spaces},
   JOURNAL = {Sci. China Math.},
  FJOURNAL = {Science China. Mathematics},
    VOLUME = {64},
      YEAR = {2021},
    NUMBER = {6},
     PAGES = {1127--1156},
      ISSN = {1674-7283},
   MRCLASS = {11F70 (11F66)},
  MRNUMBER = {4268887},
       DOI = {10.1007/s11425-018-1616-0},
       URL = {https://doi.org/10.1007/s11425-018-1616-0},
}

@article {Kobayashi:Mano,
    AUTHOR = {Kobayashi, Toshiyuki and Mano, Gen},
     TITLE = {The {S}chr\"{o}dinger model for the minimal representation of the
              indefinite orthogonal group {${\rm O}(p,q)$}},
   JOURNAL = {Mem. Amer. Math. Soc.},
  FJOURNAL = {Memoirs of the American Mathematical Society},
    VOLUME = {213},
      YEAR = {2011},
    NUMBER = {1000},
     PAGES = {vi+132},
      ISSN = {0065-9266},
      ISBN = {978-0-8218-4757-2},
   MRCLASS = {22E30 (22E46 43A80)},
  MRNUMBER = {2858535},
MRREVIEWER = {Takeshi Kawazoe},
       DOI = {10.1090/S0065-9266-2011-00592-7},
       URL = {https://doi.org/10.1090/S0065-9266-2011-00592-7},
}

@article {GK:cone,
    AUTHOR = {Gurevich, Nadya and Kazhdan, David},
     TITLE = {Fourier transform on a cone and the minimal representation of
              even orthogonal group},
   JOURNAL = {Israel J. Math.},
  FJOURNAL = {Israel Journal of Mathematics},
    VOLUME = {266},
      YEAR = {2025},
    NUMBER = {1},
     PAGES = {99--130},
      ISSN = {0021-2172},
   MRCLASS = {22E50 (20G05 20G25 43A65)},
  MRNUMBER = {4905579},
       DOI = {10.1007/s11856-025-2753-y},
       URL = {https://doi.org/10.1007/s11856-025-2753-y},
}

@article{Tr,
  title={Smith theory and geometric {H}ecke algebras},
  author={Treumann, David},
  journal={arXiv preprint arXiv:1107.3798, to appear in Math. Ann.},
  year={2011}
}

@article {B,
    AUTHOR = {Bushnell, Colin J.},
     TITLE = {Representations of reductive {$p$}-adic groups: localization
              of {H}ecke algebras and applications},
   JOURNAL = {J. London Math. Soc. (2)},
  FJOURNAL = {Journal of the London Mathematical Society. Second Series},
    VOLUME = {63},
      YEAR = {2001},
    NUMBER = {2},
     PAGES = {364--386},
      ISSN = {0024-6107},
   MRCLASS = {22E50},
  MRNUMBER = {1810135},
MRREVIEWER = {Bertrand Lemaire},
       URL = {https://doi.org/10.1017/S0024610700001885},
}

@article {F,
    AUTHOR = {Flicker, Yuval Z.},
     TITLE = {Automorphic forms on covering groups of {${\rm GL}(2)$}},
   JOURNAL = {Invent. Math.},
  FJOURNAL = {Inventiones Mathematicae},
    VOLUME = {57},
      YEAR = {1980},
    NUMBER = {2},
     PAGES = {119--182},
      ISSN = {0020-9910},
   MRCLASS = {10D40 (22E55)},
  MRNUMBER = {567194},
MRREVIEWER = {Wen Ch'ing Winnie Li},
       URL = {https://doi.org/10.1007/BF01390092},
}

@article {I,
    AUTHOR = {Ikeda, Tamotsu},
     TITLE = {On the theory of {J}acobi forms and {F}ourier-{J}acobi
              coefficients of {E}isenstein series},
   JOURNAL = {J. Math. Kyoto Univ.},
  FJOURNAL = {Journal of Mathematics of Kyoto University},
    VOLUME = {34},
      YEAR = {1994},
    NUMBER = {3},
     PAGES = {615--636},
      ISSN = {0023-608X},
   MRCLASS = {11F30 (11F27 11F55)},
  MRNUMBER = {1295945},
MRREVIEWER = {Rolf Berndt},
       URL = {https://doi.org/10.1215/kjm/1250518935},
}

@article {K,
    AUTHOR = {Kudla, Stephen S.},
     TITLE = {On the local theta-correspondence},
   JOURNAL = {Invent. Math.},
  FJOURNAL = {Inventiones Mathematicae},
    VOLUME = {83},
      YEAR = {1986},
    NUMBER = {2},
     PAGES = {229--255},
      ISSN = {0020-9910},
   MRCLASS = {22E50 (11F27 11F70)},
  MRNUMBER = {818351},
MRREVIEWER = {Marie-France Vign\'eras},
       URL = {https://doi.org/10.1007/BF01388961},
}

@article {R,
    AUTHOR = {Rallis, S.},
     TITLE = {On the {H}owe duality conjecture},
   JOURNAL = {Compositio Math.},
  FJOURNAL = {Compositio Mathematica},
    VOLUME = {51},
      YEAR = {1984},
    NUMBER = {3},
     PAGES = {333--399},
      ISSN = {0010-437X},
   MRCLASS = {22E55 (22E50)},
  MRNUMBER = {743016},
MRREVIEWER = {Martin L. Karel},
       URL = {http://www.numdam.org/item?id=CM_1984__51_3_333_0},
}

@article {Re,
    AUTHOR = {Reich, Ryan Cohen},
     TITLE = {Twisted geometric {S}atake equivalence via gerbes on the
              factorizable {G}rassmannian},
   JOURNAL = {Represent. Theory},
  FJOURNAL = {Representation Theory. An Electronic Journal of the American
              Mathematical Society},
    VOLUME = {16},
      YEAR = {2012},
     PAGES = {345--449},
      ISSN = {1088-4165},
   MRCLASS = {22E57 (11R39 14M15 20G05)},
  MRNUMBER = {2956088},
MRREVIEWER = {B. Sury},
       URL = {https://doi.org/10.1090/S1088-4165-2012-00420-4},
}

@article{Y,
  title={The {C}{A}{P} representations indexed by {H}ilbert cusp forms},
  author={Yamana, Shunsuke},
  journal={arXiv preprint arXiv:1609.07879},
  year={2016}
}

@article{C,
	title={Fourier Coefficients for {Theta} Representations on Covers of General Linear Groups},
	author={Cai, Yuanqing},
	journal={arXiv:1602.06614},
	year={2016}
}

@article{G,
  title={On Certain Global Constructions of Automorphic Forms Related to Small Representations of {$\mathrm{F}_4$}},
  author={Ginzburg, David},
  journal={arXiv:1503.06409},
  year={2015}
}

@article {Elazar:Shaviv,
    AUTHOR = {Elazar, Boaz and Shaviv, Ary},
     TITLE = {Schwartz functions on real algebraic varieties},
   JOURNAL = {Canad. J. Math.},
  FJOURNAL = {Canadian Journal of Mathematics. Journal Canadien de
              Math\'{e}matiques},
    VOLUME = {70},
      YEAR = {2018},
    NUMBER = {5},
     PAGES = {1008--1037},
      ISSN = {0008-414X},
   MRCLASS = {14P05 (14P20 22E45 46A11 46F05)},
  MRNUMBER = {3831913},
MRREVIEWER = {Zbigniew Szafraniec},
       DOI = {10.4153/CJM-2017-042-6},
       URL = {https://doi.org/10.4153/CJM-2017-042-6},
}

@article{Getz:Hsu:Leslie,
    AUTHOR = {Getz, Jayce R. and Hsu, Chun-Hsien and Leslie, Spencer},
     TITLE = {Harmonic analysis on certain spherical varieties},
   JOURNAL = {J. Eur. Math. Soc. (JEMS)},
  FJOURNAL = {Journal of the European Mathematical Society (JEMS)},
    VOLUME = {27},
      YEAR = {2025},
    NUMBER = {2},
     PAGES = {433--541},
      ISSN = {1435-9855},
   MRCLASS = {11F70 (11F55 11F85 43A32)},
  MRNUMBER = {4859573},
       DOI = {10.4171/jems/1381},
       URL = {https://doi.org/10.4171/jems/1381},
}

@article {SW:minimal,
    AUTHOR = {Savin, Gordan and Woodbury, Michael},
     TITLE = {Structure of internal modules and a formula for the spherical
              vector of minimal representations},
   JOURNAL = {J. Algebra},
  FJOURNAL = {Journal of Algebra},
    VOLUME = {312},
      YEAR = {2007},
    NUMBER = {2},
     PAGES = {755--772},
      ISSN = {0021-8693},
   MRCLASS = {20G05 (20G25 22E50)},
  MRNUMBER = {2333183},
MRREVIEWER = {Naohisa Shimomura},
       DOI = {10.1016/j.jalgebra.2007.01.014},
       URL = {https://doi.org/10.1016/j.jalgebra.2007.01.014},
}

@article {Getz:Quad,
    AUTHOR = {Getz, Jayce R.},
     TITLE = {Summation formulae for quadrics},
   JOURNAL = {Selecta Math. (N.S.)},
  FJOURNAL = {Selecta Mathematica. New Series},
    VOLUME = {31},
      YEAR = {2025},
    NUMBER = {2},
     PAGES = {Paper No. 41, 42},
      ISSN = {1022-1824},
   MRCLASS = {11F70 (11E12 11F27)},
  MRNUMBER = {4887846},
       DOI = {10.1007/s00029-025-01036-7},
       URL = {https://doi.org/10.1007/s00029-025-01036-7},
}

@ARTICLE{GK:auto,
       author = {{Gurevich}, Nadya and {Kazhdan}, David},
        title = "{Automorphic functionals for the minimal representations of groups of type $D_n$ and $E_n$}",
      journal = {arXiv e-prints},
         year = 2024,
        month = mar,
          eid = {arXiv:2403.19640},
        pages = {arXiv:2403.19640},
          doi = {10.48550/arXiv.2403.19640},
archivePrefix = {arXiv},
       eprint = {2403.19640},
 primaryClass = {math.RT},
       adsurl = {https://ui.adsabs.harvard.edu/abs/2024arXiv240319640G}
}

@article {KR:Siegel-Weil,
    AUTHOR = {Kudla, Stephen S. and Rallis, Stephen},
     TITLE = {A regularized {S}iegel-{W}eil formula: the first term
              identity},
   JOURNAL = {Ann. of Math. (2)},
  FJOURNAL = {Annals of Mathematics. Second Series},
    VOLUME = {140},
      YEAR = {1994},
    NUMBER = {1},
     PAGES = {1--80},
      ISSN = {0003-486X},
   MRCLASS = {11F70 (11F27 22E55)},
  MRNUMBER = {1289491},
MRREVIEWER = {Colette M\oe glin},
       DOI = {10.2307/2118540},
       URL = {https://doi.org/10.2307/2118540},
}

@ARTICLE{Hsu:asymptotics,
       author = {{Hsu}, Chun-Hsien},
        title = "{Asymptotics of Schwartz functions}",
      journal = {arXiv e-prints},
         year = 2021,
        month = dec,
          eid = {arXiv:2112.02403},
          pages = {arXiv:2112.02403},
          doi = {10.48550/arXiv.2112.02403},
archivePrefix = {arXiv},
       eprint = {2112.02403},
 primaryClass = {math.NT},
       adsurl = {https://ui.adsabs.harvard.edu/abs/2021arXiv211202403H}
}

@article {minrep:real,
    AUTHOR = {Hilgert, Joachim and Kobayashi, Toshiyuki and M\"{o}llers, Jan},
     TITLE = {Minimal representations via {B}essel operators},
   JOURNAL = {J. Math. Soc. Japan},
  FJOURNAL = {Journal of the Mathematical Society of Japan},
    VOLUME = {66},
      YEAR = {2014},
    NUMBER = {2},
     PAGES = {349--414},
      ISSN = {0025-5645},
   MRCLASS = {22E45 (17C30 33C10)},
  MRNUMBER = {3201818},
MRREVIEWER = {Zhanqiang Bai},
       DOI = {10.2969/jmsj/06620349},
       URL = {https://doi.org/10.2969/jmsj/06620349},
}

@book {Highertrans:II,
    AUTHOR = {Erd\'{e}lyi, Arthur and Magnus, Wilhelm and Oberhettinger,
              Fritz and Tricomi, Francesco G.},
     TITLE = {Higher transcendental functions. {V}ol. {II}},
      NOTE = {Based on notes left by Harry Bateman,
              Reprint of the 1953 original},
 PUBLISHER = {Robert E. Krieger Publishing Co., Inc., Melbourne, FL},
      YEAR = {1981},
     PAGES = {xviii+396},
      ISBN = {0-89874-069-X},
   MRCLASS = {33-02 (01A75)},
  MRNUMBER = {698780},
}

@article {Ikeda:SW,
    AUTHOR = {Ikeda, Tamotsu},
     TITLE = {On the residue of the {E}isenstein series and the
              {S}iegel-{W}eil formula},
   JOURNAL = {Compositio Math.},
  FJOURNAL = {Compositio Mathematica},
    VOLUME = {103},
      YEAR = {1996},
    NUMBER = {2},
     PAGES = {183--218},
      ISSN = {0010-437X,1570-5846},
   MRCLASS = {11F27 (11E45 11F46)},
  MRNUMBER = {1411571},
MRREVIEWER = {Rolf\ Berndt},
       URL = {http://www.numdam.org/item?id=CM_1996__103_2_183_0},
}

@article {Ichino:SW,
    AUTHOR = {Ichino, Atsushi},
     TITLE = {On the regularized {S}iegel-{W}eil formula},
   JOURNAL = {J. Reine Angew. Math.},
  FJOURNAL = {Journal f\"{u}r die Reine und Angewandte Mathematik. [Crelle's
              Journal]},
    VOLUME = {539},
      YEAR = {2001},
     PAGES = {201--234},
      ISSN = {0075-4102,1435-5345},
   MRCLASS = {11F27 (11F46 11F70)},
  MRNUMBER = {1863861},
MRREVIEWER = {Solomon\ Friedberg},
       DOI = {10.1515/crll.2001.076},
       URL = {https://doi.org/10.1515/crll.2001.076},
}
\bibliographystyle{alpha}

\end{document}